\documentclass{siamart220329}

\usepackage{mathtools}
\usepackage{amsfonts,latexsym,amssymb,amscd,epsf}
\usepackage[utf8]{inputenc}
\usepackage{bbm}
\usepackage{tikz}
\usepackage{verbatim}
\usepackage{subcaption}
\usepackage{xargs}                      

\usepackage{algorithm}
\usepackage{algpseudocode}
\algnewcommand{\LineComment}[1]{\Statex \(\%\) \small \textit{#1} \(\%\)}

\DeclareMathAlphabet{\mathdutchcal}{U}{dutchcal}{m}{n}
\SetMathAlphabet{\mathdutchcal}{bold}{U}{dutchcal}{b}{n}
\DeclareMathAlphabet{\mathdutchbcal}{U}{dutchcal}{b}{n}
\newtheorem{remark}[theorem]{Remark}
\newtheorem{example}[theorem]{Example}

 \newcommand{\I}{\mathbb{I}}
 \newcommand{\N}{\mathbb{N}}
 \newcommand{\bigO}{\mathcal{O}}
\usepackage{tensor}

\usepackage{graphicx}

\usepackage{pgfplots,relsize}
\pgfplotsset{compat=1.14}
\graphicspath{ {./figures/} }
\usepackage{tikz}
\usetikzlibrary{hobby, shapes}
\usetikzlibrary{quotes,angles,positioning}
\usetikzlibrary{calc}
\usetikzlibrary{matrix}
\usetikzlibrary{patterns}

\pgfplotsset{grid style={line width=0.05pt, gray}}
\pgfplotsset{minor grid style={gray}}
\pgfplotsset{major grid style={gray}}

\DeclareMathOperator*{\argmax}{arg max}

\newcommand{\hypref}[2]{\hyperref[#2]{#1 \ref*{#2}}}

\usepackage{etoolbox} 

\usepackage{extramath}

\definecolor{matlabred}{rgb}{0.9047,    0.1918,    0.1988}
\definecolor{matlabblue}{rgb}{0.2941    0.5447    0.7494}
\definecolor{matlabgreen}{rgb}{	0.3718    0.7176    0.3612}
\definecolor{matlaborange}{rgb}{1.0000    0.5482    0.1000}
\usepackage{todonotes}
\newcommand{\mat}[1]{\ensuremath{{\bm{{#1}}}}}

\renewcommand{\vec}[1]{\underline{\bm{#1}}}

\newcommand{\ttvec}[1]{\mathdutchbcal{#1}}
\newcommand{\ttmat}[1]{\bm{\mathcal{#1}}}

\newcommand{\sssd}{{{\sc s}{\sc \small s}--{\sc {sd}}}}
\newcommand{\sscg}{{{\sc s}{\sc \small s}--{\sc {cg}}}}

\newcommand{\tucksscg}{{\sc Tk}--\sscg}
\newcommand{\tucksssd}{{\sc Tk}--\sssd}

\newcommand{\pfft}{{\sc P}--{\sc\small FFT}}
\newcommand{\peig}{{\sc P}--{\sc\small E}{\sc \small ig}}
\newcommand{\pio}{{\sc P}--{\sc\small I}{\sc \small nn}{\sc\small O}{\sc \small ut}}

\headers{Subspace gradient descent for linear tensor equations}{M. Iannacito, 
L. Piccinini and V. Simoncini}
\title{Subspace gradient descent methods for linear tensor equations\thanks{Version of 
February 24, 2026.}}
\author{Martina Iannacito\thanks{Dipartimento di Matematica and (AM)$^2$,
		Alma Mater Studiorum Universit\`a di Bologna,
		Piazza di Porta San Donato  5, I-40127 Bologna, Italy,
		{\tt \{martina.iannacito\}\{lorenzo.piccinini12\}\{valeria.simoncini\}@unibo.it}}\and Lorenzo Piccinini$^*$
	\and Valeria Simoncini$^*$\thanks{IMATI-CNR, Pavia, Italy. }}

\ifpdf
\hypersetup{pdftitle={Tensor-SSCG}}
\fi

\begin{document}
	
	\maketitle	
	
	\begin{abstract}
The numerical solution of algebraic tensor equations
is a largely open and challenging task.
Assuming that the operator is symmetric and positive definite, we propose two
new gradient-descent type methods for tensor equations that 
generalize the recently proposed Subspace Conjugate Gradient (\textsc{S{\small s}--{cg}}), 
D. Palitta et al, SIAM J. Matrix Analysis and Appl (2025).
As our interest is mainly in a modest number of tensor modes,
the Tucker format is used to efficiently represent low-rank tensors. Moreover,
mixed-precision strategies are employed in certain subtasks to
improve the memory usage, and different preconditioners are applied to enhance
convergence.
The potential of our strategies is illustrated by experimental results on
tensor-oriented discretizations of three-dimensional partial differential equations
with separable coefficients.
Comparisons with  the state-of-the-art {\it Alternating Minimal Energy} (AMEn) algorithm
confirm the competitiveness of the proposed strategies.
\end{abstract}

%
%
%

\section{Introduction}
We are interested in solving the tensor equation 
\begin{equation} \label{eq:Ax=c}
	\ttmat{A}(\ttvec{x}) = \ttvec{c}
\end{equation}
 where $\ttmat{A}$ is a sum of tensorized operators acting on  
$\R^{n_1\times \cdots \times n_d}$, and is symmetric and positive definite, in
a sense that will be defined later; moreover, 
$\ttvec{x}$ and $\ttvec{c}$ are tensors of appropriate order and size. 
As an example, in the case $d=3$, the Kronecker form of (\ref{eq:Ax=c}) reads
$$
(\sum_{i=1}^\ell \mat{C}_i\otimes \mat{B}_i \otimes \mat{A}_i ) \vec{x} = \vec{c},
$$
where $\mat{A}_i, \mat{B}_i, \mat{C}_i$ are matrices and $\vec{x}$, 
$\vec{c}$ 
are vectors of conforming dimension.

Tensor equations naturally appear in several domains such as quantum 
chemistry~\cite{Lubich2008,Meyer2009}, financial mathematics~\cite{SLOAN19981, Wang2005}, 
imaging~\cite{Kilmer2013}, 
deep neural networks~\cite{Ji2019MLten}; we refer to~\cite{GrasedyckKT2013GAMM} for additional
established applications. 
Tensor equations  also arise  whenever certain discretizations of (parameterized)
partial differential equations
of dimension three or higher are employed,  see, e.g.,
\cite{Simoncini.survey.16, LOLI20202586, Kressner2011}.
	
Typical solvers for tensor equations
use a low-rank format to represent the objects involved in the computations. 
Among these, we recall the Tucker decomposition~\cite{Tucker1966}, the Canonical 
Polyadic Decomposition (CP)~\cite{Hitchcock1927cpd, Kolda2009TD}, the Hierarchical Tucker 
(HT)~\cite{Hackbusch2009HT, Grasedyck2009HT}, and the Tensor-Train 
(TT)~\cite{White1992MPS, Oseledets2011}. 
Moreover, methods can be grouped in two main categories.
In the first class, we include all methods that formulate the equation as an optimization problem. The 
solution is assumed to belong to a specific low-rank tensor set, whose specific constraints are used to design an optimization strategy. For instance, ALS-based methods optimize each factor of the chosen 
low-rank representation of $\ttvec{x}$, while keeping the others constant. 
Among widely used methods we recall 
the Density Matrix Re-orthogonalization Group (DMRG) method~{\cite{White1992MPS}} and its variants 
such as MALS~\cite{Holtz2012MALS}, and Riemannian optimization techniques
such as~\cite{NIPS2011_37bc2f75}, whose low-rank sets are Riemannian manifolds.
The second family includes iterative methods adapted from classical vector iterations. 
Examples of such algorithms are power methods combined with Canonical Polyadic 
tensor decomposition~\cite{Beylkin2005}, Krylov based methods combined with different 
low-rank tensor representations, for example CG and BiCG with 
H-Tucker~\cite{Kressner2011,Tobler2012}, and GMRES with 
TT~\cite{Dolgov2013TTGMRES, doi:10.1137/24M1694999}, 
see~\cite{GrasedyckKT2013GAMM} for a richer account.
In between these two classes, there is the Alternating Minimal Energy (AMEn) 
algorithm~\cite{Dolgov2014AMEn} that merges the ALS scheme with an algebraic solver, all in
 the TT format. This method is one of the most robust and efficient
solvers for tensor equations both in the symmetric and nonsymmetric cases.
%
Existing optimization-based solvers, such as AMEn, may experience lack of convergence 
guarantees and the risk of swamping in local minima, whereas classical 
iterative solvers usually offer a more transparent management of memory, 
accuracy control and possibilities for parallelization.

In this work, we propose to combine the idea of subspace gradient-descent strategies 
(such as in the matrix algorithm \sscg ~\cite{Palitta2025SsCG}) with the Tucker format to
attack the symmetric tensor equation. The matrix subspace approach developed in \cite{Palitta2025SsCG}
generates a gradient-descent type recurrence for the approximate solution, in which the
coefficient of the new direction is determined by subspace projection.
  In our setting, the choice of the Tucker format is natural for two reasons. First, when
 dealing with few modes,
the Tucker representation is quite competitive with respect to other formats that can handle massive
numbers of modes, such as Tensor Train. Second, this representation makes it easier to identify 
the projection subspaces that are used  to update the sequence iterates.
We derive two closely related approaches, which may be viewed as the subspace tensor oriented
versions of the steepest descent and conjugate gradient methods, from which we borrow the names.
Our methods are equipped with various features to improve performance, such as preconditioning,
mixed-precision computations and inexact inner solves.

Preliminary numerical experiments on various three-dimensional discretizations by finite
differences of elliptic self-adjoint partial differential equations with
separable coefficients are reported,
to illustrate the competitiveness of our newly designed solvers with respect to the 
well established AMEn method.

The manuscript is organized as follows. 
In section~\ref{sec:prelim} the employed notation for operations with tensors is
introduced.
Section~\ref{sec:subspace-CG} and its subsections
 describe the derivation of the class
of methods in tensor format. 
Section~\ref{sec:Impl} describes the implementation, first 
giving details on the use of the Tucker format within the iterative solver, and
then formalizing the computation of the recurrence coefficients. Section~\ref{sec:final} 
presents the final algorithms; its subsections describe the use of
different preconditioning strategies. Section~\ref{sec:experiments} reports
 on our experimental experience with the new algorithms. Finally, section~\ref{sec:conclusions} contains some conclusions.

\subsection{Notation and preliminaries}\label{sec:prelim}
Low-case, underlined bold letters denote vectors ($\vec{x}$), while capital, 
bold letters ($\mat{A}$) denote matrices.
Low-case, calligraphic, bold letters ($\ttvec{x}$) stand for order-$d$ tensors, either in full or in compressed format, of size $({n}_1,\dots, n_d)$. 
To shorten the notation, we say that order-$d$ tensors belong to $\R^{\vec{n}}$ where $\vec{n} = (n_1,\dots, n_d)$. {Upper-case, calligraphic, bold letters ($\ttmat{A}$) stand for linear operators among tensor spaces.} The identity matrix of dimension $n\times n$ is 
denoted $\I$, as long as the size is clear from the context. 
In the following we will make great use of the Tucker format for tensors. In particular, a $d$-mode tensor
can be represented as $\ttvec{y}=\mat{U}_{1}\otimes \cdots \otimes \mat{U}_{d} \ttvec{x}$ where $\ttvec{x}$ is the core 
tensor \cite{Tucker1966, BallardKolda2025}.
The mode-$k$ tensor-times-matrix product between a matrix 
$\mat{U}\in\R^{m_k\times n_k}$ and a tensor $\ttvec{x}\in\R^{\vec{n}}$ is a mode-wise 
tensor-times-matrix (TTM) multiplication denoted as $\ttvec{y} = \ttvec{x}\times_k\mat{U}$, 
where {$\ttvec{y}(\vec{i}) = 
\sum_{j_k=1}^{n_k}\mat{U}(i_k, j_k)\ttvec{x}(i_1, \dots, j_k, \dots, i_d)$}.

Let $\ttvec{x}$ be an order-$d$ tensor of size $\vec{n}$, and $\ttmat{A}$ a multilinear operator between tensor spaces of compatible order and size.
As a linear operator between two vector spaces (tensor space of order-$1$) can be represented by a matrix (tensor of order-$2$), a multilinear operator between tensor spaces of order $d$ can be represented by a tensor of order $2d$. 
Hence, the image of $\ttvec{x}$ through $\ttmat{A}$ corresponds 
to the contraction between the last $d$ modes of $\ttmat{A}$ and all the $d$ modes of $\ttvec{x}$, that is
\begin{equation} \label{eq:Ax}
	\left(\ttmat{A}(\ttvec{x})\right)(\vec{i}) = 			(\ttmat{A}\ttvec{x})(\vec{i}) = \sum_{\vec{j} = \bm{1}}^{\vec{n}} \ttmat{A}(\vec{i}, \vec{j}) \ttvec{x}(\vec{j}) ,
\end{equation}
with $\vec{i}= (i_1,\dots, i_d)$ and $\vec{j} = (j_1,\dots, j_d)$ multilinear indices. This is the standard way to generalize the matrix-vector product to the tensor framework. Similarly we can define the composition of two multilinear operators $\ttmat{A}$ and $\ttmat{P}$ acting on $\R^{\vec{n}}$, that is $\ttmat{A}(\ttmat{P}(\ttvec{x}))$ as the contraction between the last $d$ modes of $\ttmat{A}$ and all the $d$ modes of $\ttmat{P}(\ttvec{x})$ computed as in Equation~\eqref{eq:Ax}, that is
\begin{equation}
	\label{eq:APx}
\bigl(\ttmat{A}\bigl(\ttmat{P}(\ttvec{x})\bigr)\bigr)(\vec{i})= 	\bigl(\ttmat{A}\ttmat{P}\ttvec{x}\bigr)(\vec{i}) = \sum_{\vec{j},\vec{k}=1}^{\vec{n}}\ttmat{A}(\vec{i}, \vec{j})\ttmat{P}(\vec{j}, \vec{k})\ttvec{x}(\vec{k}).
\end{equation}
To simplify the exposition, we will write $\ttmat{A}\ttvec{x}$ and $\ttmat{A}\ttmat{P}\ttvec{x}$, referring to the operations in~\eqref{eq:Ax} and~\eqref{eq:APx}. Additionally, we assume that $\ttmat{A}$ is symmetric and positive definite, that is, $\ttmat{A}(\vec{i}, \vec{j}) = \ttmat{A}(\vec{j}, \vec{i})$ for whatever multilinear indices $\vec{i}$ and $\vec{j}$, and $\langle \ttvec{x}, \ttmat{A}\ttvec{x}\rangle > 0$, where the inner product between tensors is defined as
\[
\langle \ttvec{x} , \ttvec{y} \rangle := \sum_{\vec{i} = \bm{1}}^{\vec{n}} \ttvec{x}(\vec{i}) \ttvec{y}(\vec{i}).
\]

\section{Subspace gradient descent methods for 
multilinear system}\label{sec:subspace-CG}

With the operations formalized above,
we aim at finding the solution $\ttvec{x}\in\R^{\vec{n}}$ to the multilinear system in (\ref{eq:Ax=c})
 using an iterative method.

We follow the classical derivation of gradient descent methods as a procedure to approximate the minimizer  of a convex function. 
Let $\Phi:\R^{\vec{n}}\rightarrow\R$ be defined as
\begin{equation} \label{eq:Phi:X}
	\Phi(\ttvec{x}) =
	\frac 1 2 \langle \ttvec{x}, \ttmat{A}\ttvec{x}\rangle -
	\langle \ttvec{x}, \ttvec{c} \rangle,
\end{equation}
where $\ttmat{A}\ttvec{x}$ is computed as in \eqref{eq:Ax}, and
consider the following minimization problem: Find $\ttvec{x}^*\in\R^{\vec{n}}$ such that
$$
\ttvec{x}^* ={\rm arg} \min_{\ttvec{x}\in\R^{\vec{n}}}\Phi(\ttvec{x}).
$$
%
Starting with a zero initial guess $\ttvec{x}_0\in\R^{\vec{n}}$ and the multilinear operator 
$\ttmat{P}_0:\R^{\vec{s}_k}\rightarrow\R^{\vec{n}}$,
we define the recurrence $\{\ttvec{x}_k\}_{k\ge 0}$
of approximate solutions by means of the following relation
\begin{equation} 
\label{0eq:update:X}
	\ttvec{x}_{k+1} = \ttvec{x}_k + 
	\ttmat{P}_k\bm{\alpha}_k,
\end{equation}
where $\bm{\alpha}_k\in\R^{\vec{s}_k}$ and $\ttmat{P}_k\bm{\alpha}_k$ is computed as in \eqref{eq:Ax}.
We emphasize that $\vec{s}_{k}$, the 
vector of dimensions of the last $d$ modes of $\ttmat{P}_k$, depends on the iteration index $k$ and it may (and will) 
change as the iterations proceed, possibly growing
up  to a certain maximum value, corresponding
to the maximum allowed rank of all iterates.

The corresponding recurrence for the residual is
$$\ttvec{r}_{k+1} = \ttvec{c} - \ttmat{A}\ttvec{x}_{k+1} =\ttvec{c} - \ttmat{A}(\ttvec{x}_{k} + \ttmat{P}_k\bm{\alpha}_k) = \ttvec{r}_k - \ttmat{A}\ttmat{P}_k\bm{\alpha}_k,$$
where $\ttmat{A}\ttmat{P}_k\bm{\alpha}_k$ is computed as in \eqref{eq:APx}.
We require that the tensor operator $\ttmat{P}_k$
satisfies a descent direction
condition completely conforming  to the vector case, that is
\begin{equation}\label{eqn:descent}
	\langle\nabla\Phi(\ttvec{x}_k), \ttmat{P}_k\rangle < 0 .
\end{equation}
To determine $\bm{\alpha}_k$, we
let $\phi(\bm{\alpha}) = \Phi(\ttvec{x}_{k} + \ttmat{P}_k\bm{\alpha})$.
For given $\ttvec{x}_{k}, \ttmat{P}_k$,
at the $k$th iteration we construct $\bm{\alpha}_k$ so that
\begin{equation} \label{eq:min:alpha}
	\phi(\bm{\alpha}_k)=
	\min_{\bm{\alpha}\in\R^{\vec{s}_k}}\phi(\bm{\alpha}) .
\end{equation}



The minimizer $\bm{\alpha}_k$ can be explicitly determined by
solving a multilinear equation of reduced dimensions,
as
described in the following proposition.

\begin{proposition} 
\label{prop:alpha}
The minimizer $\bm{\alpha}_k\in\R^{\vec{s}_k}$ of~\eqref{eq:min:alpha}
is the unique solution of
\begin{equation} \label{eq:prop:alpha:thesis}
		\ttmat{P}_k^{\T}\ttmat{A}\ttmat{P}_k\bm{\alpha} =
		\ttmat{P}_k^{\T}\ttvec{r}_k ,
\end{equation}
where for $\vec{\ell}, \vec{h}\in\R^{\vec{s}_k}$ and $\vec{i}, \vec{j}\in\R^{\vec{n}}$ multilinear indices,
\begin{equation} \label{eq:PAP}
\bigl(\ttmat{P}_k^{\T}\ttmat{A}\ttmat{P}_k\bm{\alpha}\bigr)(\vec{\ell}) = 
\sum_{\vec{i},\vec{j},\vec{h}=\bm{1}}^{\vec{n}}\ttmat{P}_k(\vec{i},\vec{\ell})\ttmat{A}(\vec{i},\vec{j})\ttmat{P}_k(\vec{j},\vec{h})\bm{\alpha}(\vec{h}) .
\end{equation}
\end{proposition}

\begin{proof}
We start by explicitly writing the function $\phi$, that is
$$\phi(\bm{\alpha}) =
\frac{1}{2} \langle \ttvec{x}_{k} + \ttmat{P}_{k}\bm{\alpha},
\ttmat{A}\bigl(\ttvec{x}_{k} + \ttmat{P}_{k}\bm{\alpha}\bigr)\rangle -
\langle \ttvec{x}_{k} + \ttmat{P}_{k}\bm{\alpha},\ttvec{c}\rangle.
$$
To ease the derivation, we introduce the vectorization and matricization of tensors and the tensor operators, respectively.
	More precisely, we let
$\vec{c} = \text{vec}(\ttvec{c})\in\R^{n}$ and 
$\vec{x}_k = \text{vec}(\ttvec{x}_k)\in\R^{n}$ for vectorizations, and
	\begin{equation}
		\label{eq:matAmatP}
		\mat{A} = \text{mat}(\ttmat{A})\in\R^{n\times n}\qquad\text{and}\qquad\mat{P}_k = \text{mat}(\ttmat{P}_k)\in\R^{n\times s_k}
	\end{equation}
for matrices, 
	where $n = \prod_{h=1}^{d}n_h$ and $s_k = \prod_{h=1}^{d}s_{k,h}$.
The function $\phi$ transforms into
$\hat{\phi}:\R^{s_k}\rightarrow \R$, defined as 
\[
\hat{\phi}(\vec{\alpha}) = \frac{1}{2} \langle \vec{x}_{k} + \mat{P}_{k}\vec{\alpha},\,
\mat{A}(\vec{x}_{k} + \mat{P}_{k}\vec{\alpha})\rangle -
\langle \vec{x}_{k} + \mat{P}_{k}\vec{\alpha},\vec{c}\rangle .
\]
	
	To find the stationary points of $\hat{\phi}$, we compute the
	partial derivatives of $\hat{\phi}$ with respect to $\vec{\alpha}$: this can
	be done in vector compact form, see, e.g.,
	\cite{Petersen2008}.
	We carry out this computation term by term, getting 
	$$
	\frac{\partial(\vec{x}_k^{\T}\mat{A}(\ttvec{x}_k +
		\mat{P}_{k}\vec{\alpha}))}{\partial \vec{\alpha}}
	=
	\frac{\partial(\vec{x}_{{k}}^{\T}
		\mat{A}\mat{P}_{k}\vec{\alpha})}{\partial \vec{\alpha}}
	= (\vec{x}_k^{\T}\mat{A}\mat{P}_{k})^{\T} = \mat{P}_{k}^{\T}\mat{A}^{\T}\vec{x}_k = \mat{P}_{k}^{\T}\mat{A}\vec{x}_k
	$$
	since $\ttmat{A}$ is a symmetric. Additionally,
	
		\begin{eqnarray*}
			\frac{\partial \, \bigl((\mat{P}_{k}\vec{\alpha})^\T
				\mat{A}\bigl(\vec{x}_{k} + \mat{P}_{k}\vec{\alpha}\bigr)\bigr)}{\partial \vec{\alpha}}&=&
			\frac{\partial\, \bigl(\bigl(\mat{P}_{k}\vec{\alpha}\bigr)^{\T}\mat{A}\vec{x}_{k}\bigr)}{\partial \vec{\alpha}} +\frac{\partial\bigl(\bigl(\mat{P}_{k}\vec{\alpha}\bigr)^{\T}
				\mat{A}\mat{P}_{k}\vec{\alpha}\bigr)}{\partial \vec{\alpha}} \\
			&=&
			\frac{\partial\, \bigl(\vec{\alpha}^{\T}\mat{P}_k^{\T}\mat{A}\vec{x}_{k}\bigr)}{\partial \vec{\alpha}} +\frac{\partial\bigl(\vec{\alpha}^{\T}\mat{P}_{k}^{\T}
				\mat{A}\mat{P}_{k}\vec{\alpha}\bigr)}{\partial \vec{\alpha}} \\
			&=&
			\mat{P}_{k}^{\T}\mat{A}\vec{x}_{k} +
			2\mat{P}_k^{\T}\mat{A}\mat{P}_k \vec{\alpha}.
		\end{eqnarray*}
	%
	%
	Moreover, the derivative of 
$(\mat{P}_k\vec{\alpha}\bigr)^{\T}\vec{c}$ with respect to $\vec{\alpha}$ is given
by $\mat{P}_k^{\T}\vec{c}$, see~\cite[Equations (69), (81)]{Petersen2008}.
	Letting $\vec{r}_k  = \vec{c} - \mat{A}\vec{x}_k = \text{vec}(\ttvec{r}_{k+1})$,
the final expression of the Jacobian of $\hat{\phi}$ with respect to $\vec{\alpha}$ is thus
$$
\frac{\partial \hat{\phi}(\vec{\alpha})}{\partial \vec{\alpha}}  =
	\mat{P}_k^\T \mat{A}\mat{P}_k \vec{\alpha} - \mat{P}_{k}^{\T}\vec{r}_{k}.
$$
The vectorized solution $\vec{\alpha}_k$ of \eqref{eq:prop:alpha:thesis} is a
stationary point of $\hat{\phi}$, and consequently a stationary point of $\phi$.
To ensure that $\vec{\alpha}_k$ is a minimizer, we show that the
Hessian of $\hat{\phi}$, 
$\mat{H}_k = \mat{P}_k^{\T}\mat{A}\mat{P}_k\in\R^{s_k\times s_k}$, is positive definite.
Using the hypothesis on the operator $\ttmat{A}$,
for any nonzero $\vec{y}\in\R^{s_k}$ it holds that
$\vec{y}^\T\mat{H}_k\vec{y}>0$, so that $\mat{H}_k$ is positive definite. 
\end{proof}

\begin{remark}
{\rm 
	The minimization problem in \eqref{eq:min:alpha} can also be recast in terms of an orthogonality condition. 
	Indeed, solving~\eqref{eq:min:alpha} is equivalent to imposing 
	the following \textit{subspace orthogonality condition}
	\begin{equation}
		\label{eq:Rk1orthPk}
		\ttvec{r}_{k+1}\perp\text{range}(\ttmat{P}_k)
	\end{equation}
	where $\text{range}(\ttmat{P}_k) = \{\ttvec{y}\in\R^{\vec{n}} \;\vert\; \ttvec{y} = \ttmat{P}_k\ttvec{x} \text{ for any }\ttvec{x}\in\R^{\vec{s}_k}\}$ with $\ttmat{P}_k\ttvec{x}$ computed as in Equation~\eqref{eq:Ax}. This condition can be reformulated as 
	$\langle \ttmat{P}_k\ttvec{u}, \ttvec{r}_{k+1}\rangle = 0$
	for every $\ttvec{u}\in\R^{\vec{s}_k}$, resulting in $\ttmat{P}_k^{\T}\ttvec{r}_{k+1} = \bm{0}_{\vec{s}_{k}}$, or equivalently 
		$\bm{0}_{\vec{s}_{k}} = \ttmat{P}_k^{\T}(\ttvec{r}_{k+1}) = \ttmat{P}_k^{\T}\bigl(\ttvec{c} - \ttmat{A}\ttvec{x}_{k+1}\bigr)$,
	which corresponds to~\eqref{eq:prop:alpha:thesis}.
	This condition corresponds to~\cite[Equation (3.6)]{Palitta2025SsCG}.
}
\end{remark}

We are left to define how the direction subspace is updated. This
selection will completely define the method, among gradient descent procedures.
The considered choices are described next.

\subsection{Conjugate gradients}
We define the recurrence for the directions $\ttmat{P}_k$ as
$$	
\ttmat{P}_{k+1}\bm{\gamma}_{k+1} = \ttvec{r}_{k+1} + \ttmat{P}_k\bm{\beta}_k
$$
where $\bm{\gamma}_{k+1}\in\R^{\vec{s}_{k+1}}$, and $\ttmat{P}_{k+1}\bm{\gamma}_{k+1}$ 
are computed as in~\eqref{eq:Ax}. The tensor $\bm{\beta}_{k}\in\R^{\vec{s}_{k}}$
	is obtained by imposing that the new directions $\ttmat{P}_{k+1}$ are $\ttmat{A}$-orthogonal with respect to the previous ones, that is
\[
\ttmat{P}_{k+1}\bm{\gamma}_{k+1} \perp_{\ttmat{A}} \text{range} (\ttmat{P}_k),
\] 
or, equivalently,
\begin{equation} \label{eq:beta:subsp-orth}
\ttmat{P}_k^{\T}\ttmat{A}\ttmat{P}_{k+1}\bm{\gamma}_{k+1} = \bm{0}_{\vec{s}_{k}},
\end{equation}
where $	\ttmat{P}_k^{\T}\ttmat{A}\ttmat{P}_{k+1}\bm{\gamma}_{k+1}$ is computed 
as in \eqref{eq:PAP}. Using the vectorization, \eqref{eq:beta:subsp-orth} is written as	
\begin{equation*}	
\text{mat}(\ttmat{P}_k^{\T}\ttmat{A})\text{vec}(\ttmat{P}_{k+1}\bm{\gamma}_{k+1}) = \mat{P}_k^{\T}\mat{A}\mat{P}_{k+1}\text{vec}(\bm{\gamma}_{k+1}) = \vec{0}_{s_{k}}
\end{equation*}
where $s_{k}=\prod {s}_{h,k}$, $\mat{A}$ and $\mat{P}_{k}$ are defined in Equation~\eqref{eq:matAmatP}. 
	
Inserting the expression for $\ttmat{P}_{k+1}\bm{\gamma}_{k+1}$, (\ref{eq:beta:subsp-orth}) becomes
	$\ttmat{P}^{\T}_k \ttmat{A}(\ttvec{r}_{k+1} + \ttmat{P}_k\bm{\beta}) = \bm{0}_{\vec{s}_{k}}$,
	that is,
\begin{equation} \label{eq:beta:thesis}
		\ttmat{P}_k^{\T}\ttmat{A}\ttvec{r}_{k+1}  +
		\ttmat{P}_k^{\T}\ttmat{A}\ttmat{P}_k\bm{\beta}_k = \bm{0}_{\vec{s}_{k}}
\end{equation}
	where $\ttmat{P}_k^{\T}\ttmat{A}\ttvec{r}_{k+1}$ and $\ttmat{P}_k^{\T}\ttmat{A}\ttmat{P}_k\bm{\beta}_k $ are computed as in \eqref{eq:APx} and~\eqref{eq:PAP}, respectively.
	Equation~\eqref{eq:beta:thesis} is a multilinear equation in the unknown $\bm{\beta}_k$, with
	the same multilinear operator used to
	compute $\bm{\alpha}_k$.
	Once again, 
	$\bm{\beta}_k$ is obtained
	by solving a multilinear equation of the same type as the original one, but
	with smaller dimensions, by projecting the problem orthogonally onto $\text{range}(\ttmat{P}_k)$,
	in a matrix sense.

Concerning the quality of the computed direction iterates,
the next proposition shows that the descent direction property (\ref{eqn:descent})
	is maintained.
	
	\begin{proposition}\label{Prop:Discent_direction}
		The tensor ${\ttmat{P}}_{k+1}\bm{\gamma}_{k+1}$ is a descent direction for $\bm{\gamma}_{k+1}\in\R^{\vec{s}_{k+1}}$.
	\end{proposition}
	
	\begin{proof}
		To prove that $\ttmat{P}_{k+1}\bm{\gamma}_{k+1}$ is a descent direction,
		we must show that
		\begin{equation} \label{proof:Pk1:eq:innprod}
			\langle \nabla{\Phi}(\ttvec{x}_{k+1}), \ttmat{P}_{k+1}\bm{\gamma}_{k+1}\rangle < 0 ,
		\end{equation}
		with $\Phi$ defined in \eqref{eq:Phi:X}, and $\ttmat{P}_{k+1}\bm{\gamma}_{k+1}$ computed as in~\eqref{eq:Ax}. As in Proposition~\ref{prop:alpha}, we consider the vectorized form $\hat{\Phi}$
		, setting $\vec{\gamma}_{k+1} = \text{vec}(\bm{\gamma}_{k+1})$ and $\vec{\beta}_{k} = \text{vec}(\bm{\beta}_{k})$, both belonging to $\R^{s_k}$.
		Following~\cite[Equations (69), (81)]{Petersen2008}, we compute the terms of the Jacobian of $\hat{\Phi}$ with respect to $\vec{x}_{k+1}$, yielding
		$$\frac{\partial (\vec{x}_{k+1}^\T\vec{c})}{\partial\vec{x}_{k+1}} =
		{\vec{c}}, \qquad
		\frac{\partial 
			\bigl(\vec{x}_{k+1}^\T\mat{A}\vec{x}_{k+1}\bigr)}{\partial\vec{x}_{k+1}} =
		2\mat{A}\vec{x}_{k+1}.
		$$
		Here we used the fact that $\ttmat{A}$ is symmetric.
		Hence,
		$\nabla\hat{\Phi}(\vec{x}_{k+1}) = \mat{A}\vec{x}_{k+1}
		- \ttvec{c} = -\vec{r}_{k+1}$.
The inner product of \eqref{proof:Pk1:eq:innprod} can be written in vectorized form,
and 
\begin{equation*}
\begin{split}
\langle \nabla\hat{\Phi}(\vec{x}_{k+1}), \mat{P}_{k+1}\vec{\gamma}_{k+1}\rangle &
= \langle -\vec{r}_{k+1},  \vec{r}_{k+1} +
\mat{P}_{k}\vec{\beta}_k\rangle\\
& = - \|\vec{r}_{k+1}\|_F^2 -\langle \vec{r}_{k+1},
\mat{P}_{k}\vec{\beta}_k\rangle	\\ 
&=  - \|\ttvec{r}_{k+1}\|_F^2 -\langle \ttvec{r}_{k+1},
\ttmat{P}_{k}\bm{\beta}_k\rangle = - \|\ttvec{r}_{k+1}\|_F^2 < 0 ,
\end{split}
\end{equation*}
where, by using \eqref{eq:Rk1orthPk}, we have that
		$\langle \ttvec{r}_{k+1}, \ttmat{P}_{k}\bm{\beta}_k\rangle =  0$ for $ \ttmat{P}_{k}\bm{\beta}_k$ computed as in~\eqref{eq:Ax}.
	\end{proof}
	
	

The tensor operator $\ttmat{P}_{k+1}$ has to be retrieved 
from the tensor $\ttvec{g}_{k+1} = \ttvec{r}_{k+1} + \ttmat{P}_k\bm{\beta}_{k}$ belonging to $\R^{\vec{n}}$. 
{For this purpose, we use truncation strategies associated with operations in 
Tucker format, which is further detailed in Section~\ref{sec:tucker}.}
Here we only describe how the procedure allows us to construct the next iterate  $\ttmat{P}_{k+1}$.
Let $\mat{U}_{k+1, h}$ be the $h$th factor matrix of size $(n_{h}\times s_{k+1,h})$ and $\bm{\gamma}_{k+1}$ the core tensor from the Tucker decomposition of $\ttvec{g}_{k+1}$ at multilinear rank $\vec{s}_{k+1}$. We define $\ttmat{P}_{k+1} = \mat{U}_{k+1, 1}\otimes \cdots \otimes \mat{U}_{k+1, d}$, where $\otimes$ denotes the Kronecker product, obtaining
\[
\ttvec{g}_{k+1} = \ttvec{r}_{k+1} + \ttmat{P}_k\bm{\beta}_{k} = \ttmat{P}_{k+1}\bm{\gamma}_{k+1}.
\]
By construction, $\ttmat{P}_{k+1}$ represents a multilinear operator from $\R^{\vec{s}_{k+1}}$ to $\R^{\vec{n}}$.

We note that the initial tensor $\ttmat{P}_0$ can be obtained in the same way from the Tucker decomposition of the
initial residual, written in Tucker format.
Similarly to the remark in~\cite{Palitta2025SsCG}, the core tensor $\bm{\gamma}_{k+1}$ does not play any role in the entire algorithm. As a consequence, its computation can be avoided.
Moreover, 
$\ttmat{P}_{k+1}$ can be obtained from a Tucker approximation of $\ttvec{g}_{k+1}$ rather than 
via a decomposition, prescribing a maximal value for the multilinear rank. 
The use of a small multilinear rank threshold
is crucial to maintain memory and computational costs under control.
In particular, the computation of $\bm{\alpha}_k$ and $\bm{\beta}_k$ involves solving a reduced tensor
equation of dimensions $\vec{s}_k$, which can also cause high costs for a large multilinear rank.

Given a tensor operator $\ttmat{A}$, a right-hand side $\ttvec{c}$, and an initial guess $\ttvec{x}_0$, a convergence tolerance $\texttt{tol}$ and a maximum number of iterations \texttt{maxit}, the 
steps of the Tensor conjugate gradient method are as follows.
Later on we shall refer to this method as \tucksscg.

\begin{enumerate}
\item Compute the residual $\ttvec{r}_0=\ttvec{c}-\ttmat{A}\ttvec{x}_0$ 
\item Form the directions $\ttmat{P}_0$ 
\item For $k = 1, \dots, \texttt{maxit}$
\begin{enumerate}
\item Solve \eqref{eq:prop:alpha:thesis} to find $\bm{\alpha}_k$
\item Update the iterative solution, that is $\ttvec{x}_{k+1}=\ttvec{x}_k+\ttmat{P}_k\bm{\alpha}_k$
\item Update the residual, that is 
          $\ttvec{r}_{k+1}=\ttvec{c}-\ttmat{A}\ttvec{x}_{k+1}$ 
\item If $\|\ttvec{r}_{k+1}\|\le \texttt{tol}\|\ttvec{c}\|$, then stop
\item Solve \eqref{eq:beta:thesis} to find $\bm{\beta}_k$
\item Form the tensor $\ttvec{g}_{k+1}=\ttvec{r}_{k+1}+\ttmat{P}_k\bm{\beta}_k$, and determine
$\ttmat{P}_{k+1}$
\end{enumerate}
\end{enumerate}

\subsection{Steepest descent method}
The computation of the inner coefficient tensors $\bm{\alpha}_k$ and $\bm{\beta}_{k}$ 
requires the solution of two multilinear equations of the form as the original problem, but
with smaller dimensions. 
As an alternative, we can impose the next direction tensor to be equal to the current residual
tensor, that is
$$
\ttvec{g}_{k+1}=\ttvec{r}_{k+1}
$$
from which the next $\ttmat{P}_{k+1}$ can be directly derived.
This completely avoids the computation of one of the two coefficients, namely $\bm{\beta}_{k}$. 
Note that such a choice corresponds to the classical steepest descent algorithm in the vector case.
This algorithm is expected to be asymptotically slower than the conjugate gradient method 
(\cite[section 5.3]{Saad2003}), however, if the number of extra iterations is not very significant, the
reduced cost can largely overcome the lower speed. The reported numerical experiments confirm this
expectation.
We shall refer to this method as \tucksssd.



\section{Implementation}\label{sec:Impl}
The sound implementation of both algorithms requires the description of several technical details.
Every operations with tensors such as sums and products, increase the tensor rank. As previously mentioned,
a maximum tensor rank needs to be fixed, to avoid wild memory consumptions, leading to rank truncation 
strategies. Moreover, the computation of the recurrence coefficients requires solving (small) tensor equations,
which all need be addressed. In the following we provide these details, together with the complete
algorithms.

\subsection{The Tucker format and tensor operations}\label{sec:tucker}
The Tucker decomposition \cite{Tucker1966, BallardKolda2025}
factorizes $\ttvec{x}$ as the product of a smaller order-$d$ tensor, $\ttvec{c}_{\ttvec{x}}\in\R^{\bf r}$  called \emph{core tensor}, times $d$ factor matrices, $\mat{U}_k\in\R^{n_k\times r_k}$, that is
\[
\ttvec{x} = \ttvec{c}_{\ttvec{x}}\times_1 \mat{U}_1 \times_2 \mat{U}_2 \cdots \times_d \mat{U}_d, 
\]
with $\mat{U}_j$ having orthonormal columns.
The index vector $\vec{r} = (r_1, \dots, r_d)$ is called \emph{multilinear rank}. The Tucker factorization can be equivalently expressed as 
\[
\ttvec{x} = \mat{U}_1\otimes \cdots \otimes \mat{U}_d\,\ttvec{c}_{\ttvec{x}},
\]
see~\cite[Proposition 3.22]{BallardKolda2025} for further details. Several algorithms exist for computing the Tucker decomposition, among which the most widely used are High-Order SVD (HOSVD)~\cite{Lathauwer2000HOSVD}, the Sequentially Truncated HOSVD (ST-HOSVD)~\cite{Vannieuwenhoven2012STHOSVD}, and the High-Order Orthogonal Iteration~\cite{Kroonenberg1980,Kapteyn1986,Lathauwer2000lmlra}. 
Both tensor-oriented algorithms are designed with all tensors in Tucker format. 
However, sums and products among tensors tend to be computationally expensive and memory consuming.
Thus, we employ two distinct operations to handle the sum of tensors and the 
product between a tensor operator and a tensor in a cheaper way by truncation.

\paragraph*{Sum of Tucker-tensors}
Consider the sum of two tensors $\ttvec{x}$ and $\ttvec{y}$,
both expressed in Tucker-format,  that is
\[
\ttvec{x} = \mat{F}_1\otimes \cdots \otimes \mat{F}_d\,\ttvec{c}_{\ttvec{x}}\qquad
\text{and}\qquad \ttvec{y} = \mat{G}_1\otimes \cdots \otimes \mat{G}_d\,\ttvec{c}_{\ttvec{y}},
\]
where $\ttvec{c}_{\ttvec{x}}\in\R^{\vec{r}}$ and $\ttvec{c}_{\ttvec{y}}\in\R^{\vec{p}}$. To 
construct the Tucker factorization of the sum $\ttvec{z} = \ttvec{x} + \ttvec{y}$, 
i.e., $\ttvec{z} = \mat{L}_1\otimes \cdots\otimes \mat{L}_d\,\ttvec{c}_{\ttvec{z}}$ 
we proceed as follows. 
The tensor core $\ttvec{c}_\ttvec{z}$ is an $\vec{r}+\vec{p}$ block-diagonal tensor 
computed by diagonally concatenating the core tensors of $\ttvec{x}$ and $\ttvec{y}$, 
that is ${\ttvec{c}}_{\ttvec{z}} = 
\texttt{tenblkdiag}(\ttvec{c}_\ttvec{x}, \ttvec{c}_\ttvec{y})$. 
The factor matrices of $\ttvec{z}$ are obtained by matrix superposition, that is  
${\mat{L}}_j  = [\mat{F}_j , \mat{G}_j ]\in\R^{n_j \times (r_j  + p_j )}$ for $j  = 1, \dots, d$. 
The storage costs of these quantities grows exponentially with $d$, hence
a size reduction becomes mandatory.
We then approximate $\mat{L}_j $ with a truncated-SVD 
by discarding the singular values which are less than a threshold, and the corresponding left {and right} singular
vectors. {We thus store the retained left singular matrix 
$\mat{U}_j\in\R^{n_j \times t_j}$}, 
and we multiply $\ttvec{c}_{\ttvec{z}}$ along each mode
by the corresponding singular value matrix and transposed right singular matrix.
These steps are repeated for every $j =1,\dots, d$. 
The so-obtained core tensor ${\ttvec{c}}_{\ttvec{z}}$ is truncated by ST-HOSVD
at multilinear-rank equal to $\texttt{maxrank}$ in all its components,
 giving ${\ttvec{c}}_{\ttvec{z}} = \mat{M}_1\otimes \cdots \otimes \mat{M}_d\, {\ttvec{c}}'_{\ttvec{z}}$, 
where ${\ttvec{c}}'_{\ttvec{z}}\in\R^{\vec{q}}$, and 
$\mat{M}_j \in\R^{t_j  \times q_j}$ with $q_j  \le \texttt{maxrank}$ 
for $j  = 1, \dots, d$. The final Tucker-approximation of $\ttvec{z} = \ttvec{x} + \ttvec{y}$ is
\[
{\mat{L}}'_1\otimes\cdots\otimes{\mat{L}}'_d\, {\ttvec{c}'}_{\ttvec{z}}
\]
where ${\mat{L}}'_j  = {\mat{U}}_j \mat{M}_j $ for $j =1, \dots, d$. 
This rounding strategy is summarized in the appendix as
 Algorithm~\ref{alg:tucker-round}, and it will referred to as
\texttt{Tucker-rounding-sum}; computational and memory costs are also discussed there.

\paragraph*{Product of tensor operator times Tucker-tensor}
A very similar rounding strategy has to be applied to compute the product between a tensor operator, $\ttmat{A}$, and a tensor in Tucker format, $\ttvec{x} = \mat{F}_1\otimes \cdots \otimes \mat{F}_d\,\ttvec{c}_{\ttvec{x}}$. Recalling that a tensor operator $\ttmat{A}:\R^{\vec{n}}\rightarrow\R^{\vec{m}}$ can be expressed as 
\begin{equation}
	\label{eq:tenop}
	\ttmat{A}
	= \sum_{h=1}^{\ell} \mat{A}_{1,h} \otimes \cdots\otimes \mat{A}_{d,h}
\end{equation}
where $\mat{A}_{j,h}$ is an $(m_j\times n_j)$ matrix, $\ell\in\N$ is not minimal.
The product $\ttmat{A}\ttvec{x}$ results in
\begin{equation}
\begin{split}
\ttvec{y} = \ttmat{A}\ttvec{x} &
= \Bigl(\sum_{h=1}^{\ell} \mat{A}_{1,h} \otimes \cdots\otimes 
\mat{A}_{d,h}\Bigr)\mat{F}_1\otimes \cdots \otimes \mat{F}_d\,\ttvec{c}_{\ttvec{x}}. 
\end{split}
\end{equation}
This tensor operator times Tucker-tensor product reduces to the sum of $\ell$ tensors in Tucker format. Thus, an algorithm similar to the \texttt{Tucker-rounding-sum} can be designed to efficiently compute the 
Tucker (approximate) factorization of $\ttvec{y}$ in the form
 $\mat{G}_1\otimes \cdots \otimes \mat{G}_d\,\ttvec{c}_{\ttvec{y}}$. First we initialize the core tensor $\ttvec{c}_{\ttvec{y}}$ equal to $\ttvec{c}_{\ttvec{x}}$, and we define the matrices  $\mat{G}_{j, 1} = \mat{A}_{j,1}\mat{F}_j$ for $j=1,\dots, d$. Then, for $h=2$, we update $\ttvec{c}_{\ttvec{y}}$ by diagonally concatenating it with $\ttvec{c}_{\ttvec{x}}$, and we 
form the matrix $\mat{G}_{j,h} = [\mat{G}_{j-1, h}, \mat{A}_{j,h}\mat{F}_j]$ for $j=1,\dots, d$. 
{As done with $\mat{L}_j$ in the previous paragraph, we approximate $\mat{G}_{j,h}$ by a truncated-SVD
with thresholding.
We store the retained left singular matrix $\mat{U}_j\in\R^{n_j \times t_j}$, 
and we multiply $\ttvec{c}_{\ttvec{y}}$ along mode $j$ by the corresponding 
singular value matrix and transposed right singular matrix.} 
We repeat these steps for every $j=1,\dots, d$, and then we increment $h$, until $j=\ell$. Once we have iterated over all modes and all terms of $\ttmat{A}$, the core tensor ${\ttvec{c}}_{\ttvec{y}}$ is approximated at multilinear-rank equal to $\texttt{maxrank}$ in all its component by ST-HOSVD, obtaining ${\ttvec{c}}_{\ttvec{y}} = \mat{L}_1\otimes \cdots \otimes \mat{L}_d\, {\ttvec{c}'}_{\ttvec{y}}$, where ${\ttvec{c}'}_{\ttvec{y}}\in\R^{\vec{q}}$, and $\mat{L}_j \in\R^{t_j  \times q_j}$ with $q_j  \le \texttt{maxrank}$ for $j  = 1, \dots, d$. 
The final Tucker-approximation of $\ttvec{y} = \ttmat{A}\ttvec{x}$ is
${\ttvec{y}}= {\mat{G}}'_1\otimes\cdots\otimes{\mat{G}}'_d\, {\ttvec{c}'}_{\ttvec{y}}$,
where ${\mat{G}}'_j  = {\mat{U}}_j \mat{L}_j $ for $j =1, \dots, d$. 
This rounding strategy is summarized in the appendix as 
Algorithm~\ref{alg:tucker-roundAx} and will be referred to as, 
\texttt{Tucker-rounding-product}; costs are also described there.

\subsection{Computation of the tensor coefficients}
To compute the coefficients $\bm{\alpha}_k\in\R^{\vec{s}_k}$ (and 
$\bm{\beta}_k\in\R^{\vec{s}_k}$ in the conjugate gradient case), 
we need to solve a multilinear system of the same order as the original, 
but lower dimension, having $s_{k, h} \le n_h$ for $h=1, \dots, d$. 
The problem size grows exponentially with $d$ as $\prod_{h=1}^d s_{h,k}$, hence
already for modest values of both  $s_{h,k}$ and $d$, efficient solvers need be considered.
In our implementation we used the \sscg algorithm in~\cite{Palitta2025SsCG}
on the problem matricization along mode $1$ of 
\eqref{eq:prop:alpha:thesis} and \eqref{eq:beta:thesis} (the latter only for the conjugate gradients).
More precisely, at iteration $k$ we solve
\begin{equation} \label{eq:alpha:sscg}
\sum_{h=1}^{\ell}\widetilde{\mat{A}}_{1,h}\bm{\alpha}^{(1)}\bigl(\widetilde{\mat{A}}_{d,h}\otimes \cdots \otimes \widetilde{\mat{A}}_{2,h}\bigr)^{\top}  
= \widetilde{\mat{S}}_{1, k} \mat{C}_{\ttvec{r}_k}^{(1)}\bigl(\widetilde{\mat{S}}_{d, k}\otimes \cdots \otimes \widetilde{\mat{S}}_{2, k}\bigr)^{\top} 
\end{equation}
where 
$\widetilde{\mat{A}}_{j,h} = \mat{U}_j^{\top}\mat{A}_{j,h}\mat{U}_j$ and
$\widetilde{\mat{S}}_{j,k} = \mat{U}_j^{\top}\mat{S}_{j, k}\mat{U}_j$,
where $\ttvec{r}_{k} = \mat{S}_{1,k}\otimes\cdots\otimes\mat{S}_{d,k}\,\ttvec{c}_{\ttvec{r}_{k}}$.
The right matrices of the operator in these equations are obtained as the Kronecker product between 
$d-1$ matrices;
as $d=3$ in all our experiments and their size is limited, we explicitly store the resulting
matrices.
Similarly for $\bm{\beta}$, with the same coefficient tensor and different right-hand side.


\paragraph*{Mixed-precision computation}
The rounding of tensor sums and contractions is the operation that mostly affects the 
accuracy of the final iterative solution. Thus, some other operations can be performed 
with a less accurate machine precision to further reduce the computational costs. 
In particular, the computation of $\bm{\alpha}$ (and $\bm{\beta}$ in \tucksscg) are performed
in {\it single machine precision}, and so is the application of the 
\peig\ and \pfft\ preconditioners.
The resulting tensor is cast back in double precision.

\section{The complete preconditioned gradient descent algorithms}\label{sec:final}
The complete \tucksssd\ and
 \tucksscg\ strategies are outlined in Algorithm~\ref{alg:tucksscg}.

\begin{algorithm}[h]
{\begin{algorithmic}[1]
\Statex \textbf{Input:} Tensor operator $\ttmat{A}$, right-hand side $\ttvec{c}$, 
an initial guess $\ttvec{x}_0$ in Tucker format, maximum multilinear rank $\vec{r}$, 
truncation tolerance $\delta$, maximum number of iterations $\texttt{maxit}$, 
tolerance $\texttt{tol}$, preconditioner $\ttmat{M}$, method (1/2): 
Steepest Descent / Conjugate Gradients.
\Statex \textbf{Output:} Approximate solution $\ttvec{x}_k$  such that 
$\|\ttmat{A}\ttvec{x}_k-\ttvec{c}\|\leq \|\ttvec{c}\| \cdot \texttt{tol}$
\smallskip
\State Set $\ttvec{y} = \texttt{Tucker-rounding-product}(\ttmat{A}, \ttvec{x}_0, \vec{r}, \delta)$ and

 $\ttvec{r}_0=\texttt{Tucker-rounding-sum}\bigr(\ttvec{c}, \ttvec{y}, \vec{r}, \delta\bigl)$
			\State Set $[\sim, \{\mat{U}_h\}_h] = \texttt{ST-HOSVD}(\ttvec{r}_0, \vec{r})$ and $\ttmat{P}_0 = \mat{U}_1\otimes\cdots\otimes\mat{U}_d$\label{alg:tucksscg:P0}
	\For{$k=0,\ldots,\mathtt{maxit}$}
	\State Compute $\bm{\alpha}_k$ by solving \eqref{eq:alpha:sscg} with  \sscg 
			\Comment{single precision} 
			\State Set $\ttvec{x}_{k+1}=\texttt{Tucker-rounding-sum}\bigr(\ttvec{x}_k, \ttvec{w}, \vec{r}, \delta\bigl)$  where $\ttvec{w} = \ttmat{P}_k\bm{\alpha}_k$ is already in Tucker format
\State Compute $\ttvec{y} = \texttt{Tucker-rounding-product}(\ttmat{A}, \ttvec{x}_{k+1}, \vec{r}, \delta)$
\State Set $\ttvec{r}_{k+1}=\texttt{Tucker-rounding-sum}\bigr(\ttvec{c}, -\ttvec{y}, \vec{r}, \delta\bigl)$  
			\State {{\bf if} $\|\ttvec{r}_{k+1}\|\leq \|\ttvec{c}\| \cdot \texttt{tol}$} then {\tt stop}
			\State {\bf if} method=1, set $\ttvec{g}_{k+1} = \ttvec{r}_{k+1}$
\State {\bf else, if} method=2, Compute $\bm{\beta}_k$ by solving \eqref{eq:beta:thesis} with \sscg
			{\Comment{single precision}} and set \label{alg:tucksscg:beta}
			 $$
\ttvec{g}_{k+1} = \texttt{Tucker-rounding-sum}\bigr(\ttvec{r}_{k+1}, \ttvec{v}, \vec{r}, \delta\bigl)
$$ 
where $\ttvec{v} = \ttmat{P}_k\bm{\beta}_k$ is already in Tucker format\label{alg:tucksscg:gk}
			\State Set $[\sim, \{\mat{U}_h\}_h] = \texttt{ST-HOSVD}(\ttvec{g}_{k+1}, \vec{r})$ and $\ttmat{P}_{k+1} = \mat{U}_1\otimes\cdots\otimes\mat{U}_d$\label{alg:tucksscg:Pk}
			\EndFor
			\State Return $\ttvec{x}_{k+1}$
		\end{algorithmic}    \caption{\tucksssd, \tucksscg}\label{alg:tucksscg}
	}
\end{algorithm}


In the following we describe preconditioning strategies aimed at
accelerating the proposed algorithms.
Preconditioning is routinely used in vector 
gradient descent methods such as CG, and it usually consists in
the construction of a matrix or operator  to be applied to the original problem.
In multiterm equations building an effective preconditioner is particularly complex, because
of the number of operators to take care of. The presence of tensor addends just exacerbates the problem.
Modern literature has yet to establish an effective and reliable preconditioner for 
generic tensor operators, hence the topic is still exploratory.
For the sake of generality, we thus consider as preconditioner $\ttmat{M}$ 
a tensor operator from $\R^{\vec{n}}$ to  $\R^{\vec{n}}$. 
As it is usually done in the
vector case in~\cite{Golub2013MatComp}, and for the matrix case in~\cite{Palitta2025SsCG}, 
the operator can be applied directly to the residual tensor, 
that is, to the tensor computed in Line 7 of Algorithm~\ref{alg:tucksscg} 
If now we denote ${\bm{z}}_k = \ttmat{M}^{-1}\ttvec{r}_k$, the gradient descent method
defines $\ttmat{P}_{k+1}={\bm{z}}_k$, while the conjugate gradient method uses the
following the updating formula 
\[
\ttmat{P}_{k+1}{\bm{\gamma}}_{k+1} =  {\bm{z}}_k + \ttmat{P}_k\bm{\beta}_k.
\]
This way, in both algorithms
the generated direction tensors will actually be the preconditioned direction tensors. We recall here
that all these operations are performed under the truncated tensorized framework. In particular,
in Line~\ref{alg:tucksscg:P0} of  Algorithm~\ref{alg:tucksscg}, $\ttvec{r}_0$ is replaced 
by $\ttvec{z}_0 = \texttt{Tucker-rounding-product}(\ttmat{M}, \ttvec{r}_0, \vec{r}, \delta)$. 
Similarly, $\ttvec{g}_{k+1}$ in Line~\ref{alg:tucksscg:Pk} is replaced by 
$\ttvec{z}_{k+1} = \texttt{Tucker-rounding-product}(\ttmat{M}, \ttvec{g}_{k+1}, \vec{r}, \delta)$.

\subsection{Matrix preconditioning}\label{sec:prec}
To simplify the exposition, in this section we will assume to work with order-$3$ tensors. Whenever the problem tensor stems from the discretization of a multidimensional elliptic PDE,
a possibly successful preconditioner investigated in the literature
consists of explicitly approximating the inverse of the Laplace operator \cite{Hackbusch2006}.
More precisely, let $\bm{\Delta}_1 = (n+1)^2{\rm tridiag}(-1, \underline{2}, -1)$ be the matrix discretization of the 1-dimensional Laplace operator on a limited closed interval over a grid of $n$.
%
In~\cite{Hackbusch2006}, 
the authors show that an approximate inverse of the discrete $3$-dimensional 
(tensor) Laplace operator can be written as
\begin{equation}
	\label{eq:invLap}
	\ttmat{M} = 
	\sum_{h= -q}^{q}c_h\,\exp(t_h\bm{\Delta}_1)\otimes \exp(t_h\bm{\Delta}_1)\otimes\exp(t_h\bm{\Delta}_1) , 
\end{equation}
where $c_h = \eta t_h$, $t_h = \exp(h\eta)$, $\eta = \pi/\sqrt{q}$. 
We consider two possible strategies to apply the operator $\ttmat{M}$ as preconditioner.

\paragraph{Eigenvalue factorization} The first strategy explicitly forms the exponential matrices 
appearing in Equation~\eqref{eq:invLap} in a generic manner.
Recalling that $\bm{\Delta}_1$ is a symmetric, we let $\bm{\Delta}_1 = \mat{V}\bm{\Lambda}\mat{V}^{\top}$ 
be its spectral decomposition, so that
	$\exp(t_h\bm{\Delta}_1) = \mat{V}\exp(t_h\bm{\Lambda})\mat{V}^{\top}$.
The operator $\ttmat{M}$ can
be applied to a given tensor, by a sequence of tensor-tensor operations.
As the exponential matrices are equal in all modes, we have to form only $2q+1$ exponential matrices, which
is costly also for moderate sizes of $\bm{\Delta}_1$. Hence, $q$ should be kept low. In our experiments, we
used $q=1$.
The action of the operator over a tensor in Tucker format can be computed using the 
\texttt{tucker-rounding-product} routine, cf. Algorithm~\ref{alg:tucker-roundAx}.
We observe that this procedure does not exploit the structure of $\bm{\Delta}_1$, except for its symmetry,
thefore it could be applied for other symmetric matrices.
%
In the numerical experiments section below we will refer to this preconditioning strategy as
\peig.

\paragraph{Fast Fourier Transforms} The second strategy corresponds to an operator-aware version
of the previous strategy. Indeed, the full eigendecomposition of $\bm{\Delta}_1$ is analytically
available, and the exponential enjoys particularly strong structural properties when applied to Kronecker sums.
Hence, the whole procedure significantly simplifies. Let $\mat{V}\bm{\Lambda}\mat{V}^{-1}$ be the eigenvalue decomposition of $\bm{\Delta}_1$. Then, setting $\ttmat{V}_h = \mat{V}_h\otimes\mat{V}_h\otimes\mat{V}_h$, 
it holds that (\cite{Dolgov2013TTGMRES, Higham2008})
%
\begin{equation*}
	\exp(t_h\bm{\Delta}_1)\otimes \exp(t_h\bm{\Delta}_1)\otimes\exp(t_h\bm{\Delta}_1)= \ttmat{V}_h\bigl(\exp(t_h\bm{\Lambda})\otimes\exp(t_h\bm{\Lambda})\otimes\exp(t_h\bm{\Lambda})\bigr)\ttmat{V}_h^{-1}
\end{equation*}  
from which, using
$\bm{\Lambda} \oplus \bm{\Lambda}\oplus \bm{\Lambda} := \bm{\Lambda}\otimes\I \otimes\I +\I \otimes\bm{\Lambda}\otimes\I +\I\otimes\I \otimes\bm{\Lambda}$,
\begin{equation*}
	\exp(t_h\bm{\Delta}_1)\otimes \exp(t_h\bm{\Delta}_1)\otimes\exp(t_h\bm{\Delta}_1) = \ttmat{V}_h\exp\bigl(t_h(\bm{\Lambda} \oplus \bm{\Lambda}\oplus\bm{\Lambda})\bigr)\ttmat{V}_h^{-1}.
\end{equation*}
Thanks to this construction, the operator defined in~\eqref{eq:invLap} can be expressed as
\begin{equation*}
	\ttmat{M} = \sum_{h=-q}^{q}c_h \ttmat{V}_h\exp\bigl(t_h(\bm{\Lambda} \oplus \bm{\Lambda}\oplus\bm{\Lambda})\bigr)\ttmat{V}_h^{-1}
\end{equation*}
Let $\ttvec{x} = \mat{U}_1\otimes\mat{U}_2\otimes\mat{U}_3\,\ttvec{c}_{\ttvec{x}}$, then by substitution,
$\ttmat{V}_h^{-1}\ttvec{x} = 
\mat{V_h^{-1}}\mat{U}_1\otimes\mat{V_h^{-1}}\mat{U}_2\otimes\mat{V_h^{-1}}\mat{U}_3\,\ttvec{c}_{\ttvec{x}}$.
Recalling that the $\mat{V}_h$s are the eigenvectors of the 1-D Laplacian,
the matrices $\widehat{\mat{U}}_{j} = \mat{V}_h^{-1}\mat{U}_j$ correspond to the image 
of $\mat{U}_j$ by the Discrete Sine Transform of Type I (DST-I). 
Similarly, the product between $\ttmat{V}_h$ and a tensor in Tucker format can be computed using 
the inverse of DST-I (iDST-I). 
We refer to the appendix for more details on our implementation in the tensor setting.
We expect this preconditioner to be very fast in terms of application, and very effective in
reducing the number of iterations whenever the problem operator is close to the d-D Laplacian.

In our numerical experiments we shall use the term {\pfft } to denote this implementation
of the preconditioner.

\subsection{Inner-Outer preconditioning}
A classical and easy way to construct a preconditioning operator consists of
approximating the action of $\ttmat{A}^{-1}$ to a tensor by calling a few iterations,
 usually a fixed number of them, of some iterative procedure. This strategy takes
the name of inner-outer procedure: at each iteration of an outer iterative methods, 
another - or the same - iterative method is called for a few iterations.
We refer to \cite{Saad1993, Golub99} for examples in the nonsymmetric and symmetric
matrix cases, respectively, and to the 
bibliography of \cite{Simoncini2003d} for a recollection of more references.
With this strategy, instead of applying the preconditioner, 
we call a few iterations of either of the two new methods to solve the multilinear system
$\ttmat{A}{\bm z}_k = \ttvec{r}_k$.
Stopping the inner solver after a limited number of steps is key to 
ensure that the overall procedure remains computationally competitive.
This strategy is useful for those applications where an explicit preconditioner is 
unavailable or too expensive to construct.

In the experiments in Section \ref{sec:experiments}, 
we choose {\tucksssd } as inner solver because is cheaper than \tucksscg. 
We set the stopping tolerance to $10^{-1}$ and the maximum number of iterations to $4$.
This strategy is referred to as \pio.
\section{Numerical experiments}
\label{sec:experiments}
 This section is devoted to the illustration of the performance
of the proposed class of methods. 
The experiments are performed in MATLAB 2025b on a machine with an 
Intel Core i7-12700H CPU @ 2.30GHz and 16 GB of RAM. 

We consider 
the \textit{Alternating Minimal Energy} algorithm 
(AMEn)\footnote{Code available at \\
\href{https://mathworks.com/matlabcentral/fileexchange/46312-oseledets-tt-toolbox}{\texttt{https://mathworks.com/matlabcentral/fileexchange/46312-oseledets-tt-toolbox}}}~\cite{Dolgov2014AMEn} for comparison purposes. 
In AMEn data and working quantities are stored in Tensor-Train (TT) format, a
popular and effective way of representing high order tensors, as
 it requires $\mathcal{O}(dnr^2)$ memory allocations  \cite{Oseledets2011}.
For large $d$, this compares favourably with the Tucker representation, which
uses $\mathcal{O}(dnr+r^d)$ storage.
AMEn alternates between two main phases: i) a local optimization phase, in which one 
block of the TT-vector (a TT-core) is updated at a time by solving a 
small projected linear system; ii) an expansion phase, where information 
from the global residual is used to enrich the local search space before 
proceeding to the next TT-core. {The phases round-robin through
	all cores; such a cycle is called \textit{sweep}}.
We should mention that in our experiments we will consider the discretization of
three-dimensional PDEs, yielding a moderate ($d=3$) order tensor, making
the Tucker format more appealing.
Nonetheless, an implementation of our algorithms can be naturally carried out
in the TT-format in case problems with high order tensors occur, without any further
changes in the algorithms.


All experiments below deal with algebraic problems stemming from the discretization of
separable elliptic operators. More precisely,
consider the discretization by finite differences
of three-dimensional
diffusion problems with separable coefficients, defined in a box with
zero boundary conditions. As usual, the one-dimensional Laplace operator is given
by
$$
\mat{\Delta}_1 := \frac 1 {h^2} {\rm tridiag}(-1, \underline{2}, -1)
\in \mathbb{R}^{(n+1)\times(n+1)},
$$ 
where $h$ is the grid step size and $n+1$
is the number of stencil points. The diffusion 1D operator with non-constant
coefficients can be discretized as
$$
(a(x)u_x)_x \approx {\bm R} {\bm D} {\bm R}^T, \qquad 
{\bm R} := \begin{bmatrix}
    1 & -1 & & & \\
     & 1 & -1 & & \\
      &&\ddots&\ddots& \\
      &&&1&-1\\
\end{bmatrix}\in\mathbb{R}^{(n+1)\times(n+2)},
$$
and ${\bm D}={\rm diag}(a_{\frac 1 2}, a_{\frac 3 2}, \ldots, a_{n+\frac 3 2})$, with 
$a_j=a(x_j)$.
If the term does not involve a derivative, such as in $a(x)u(x,y,z)$, then the
matrix coefficient is given as
${\bm D}_x={\rm diag}(a_1, a_2, \ldots, a_{n+1})$.

A few more technical details are reported 
to describe our experimental setting.
As already mentioned,
the solution of the reduced problem in $\bm \alpha$ (and $\bm \beta$) is performed
by using \sscg\ (\cite{Palitta2025SsCG}) where two of the three modes are matricized.
The parameter settings for \sscg\ are initial guess $\texttt{x0} = 0$,
$\texttt{tol\_inner} = \min\{10^{-2}, \|\ttvec{r}_k\|/\|\ttvec{r}_0\|\}$,
maximal number of iterations $\texttt{maxit\_inner} = 100$, 
maximal rank equal to the $\texttt{maxrank}$ value prescribed for the outer solver,
maximal residual rank $d\cdot\texttt{maxrank}$.  
{In AMEn, we use the same values for the corresponding parameters, 
that is, zero as initial guess (\texttt{x0} = $0$),
 same maximal rank (\texttt{rmax} = \texttt{maxrank})
 and convergence tolerance, (\texttt{tol\_exit} = \texttt{tol}). 
 For all the numerical experiments,  we tested the three 
 preconditioning implemented options in AMEn; while they 
 are effective in reducing the number of sweeps, 
 the total CPU time is 
higher than with unpreconditioned AMEn.} 

As already mentioned, we expect 
the steepest descent method to show lower CPU time than the conjugate gradient strategy
as long as few iterations are carried out, while the opposite will occur when convergence
is slower, especially because the steepest descent method will have an asymptotically slower
convergence.
Therefore, in all experiments we report results for both methods in the unpreconditioned
case, whereas we only show results with \tucksssd\ in the preconditioned case, when 
convergence is significantly faster.

{We observed empirically that AMEn's CPU time may vary significantly, 
for example, if Matlab needs to allocate memory. 
For a fair comparison, we run all algorithms several times 
and report the average CPU time for each algorithm.}

\begin{table}[ht]
\centering
\caption{Example~\ref{subsec:ex1}. 
CPU time in seconds and number of iterations in parenthesis
for different solvers and $n=500, 1000$.  \label{tab:ex1}}
\begin{tabular}{|c|l|l|l|}
\hline
& Algorithm & $n = 500$ & $n = 1000$ \\
\hline
{\tt tol}=$10^{-3}$ &			\tucksscg        	& 0.33 (19)   	&  0.28 (19)   \\
		&			\tucksssd    &   0.24 (20)	&  0.23 (19)  \\
		&\pio\ \tucksssd 	&  0.57 (6)  &   0.60 (6)   \\
		&\pfft\ \tucksssd &  \textbf{0.06} (3)  &  \textbf{0.05} (3)   \\
		&\peig\ \tucksssd & 0.08 (2)   &   0.27 (2)      \\
		&AMEn             	& 0.23 (6*)	& 0.75 (5*)  \\
	\hline

{\tt tol}=$10^{-4}$ &			\tucksscg        	& 1.13 (48)  	&   2.16 (49) \\
		&			\tucksssd        	& 1.40 (79)  	&   3.85 (79) \\
		&\pio\ \tucksssd 	&  1.28 (15)   &   1.86 (18)  \\
		&\pfft\ \tucksssd &  \textbf{0.09} (4)  &    {\textbf{ 0.10} (5)} \\
		&\peig\ \tucksssd & 0.51 (21)  &      1.69 (31)   \\
		&AMEn             	& 0.72 (6*) 	&    	2.16 (6*)\\
			\hline
\end{tabular}\\
\begin{flushleft}
\hspace{60pt}$^*$ number of sweeps
\end{flushleft}
\end{table}

\vskip 0.1in
\begin{example} 
{\rm 
We consider the 3-dimensional Poisson equation $ -\Delta u = f$, $ u=u(x,y,z)$.
Using the definitions above for the discretized one-dimensional operators
we obtain the finite difference 
approximation of the 3-D Laplacian as
    \begin{equation}\label{subsec:ex1}
        \bm{\Delta}_3 := \mat{\Delta}_1 \otimes \I \otimes \I + 
        \I \otimes \mat{\Delta}_1 \otimes \I + \I \otimes \I \otimes \mat{\Delta}_1.
    \end{equation}
Note that the Kronecker operator has dimension $(n+1)^3\times (n+1)^3$, if $n$ nodes are used to 
discretize the second order derivative in each dimension.
    The forcing term is represented by the
tensor 
$\ttvec{f} = \vec{v}_1\otimes \vec{v}_2 
\otimes \vec{v}_3\in\mathbb{R}^{\vec{n}+\vec{1}}$,
where $\vec{v}_1$ is the normalized vector of all ones, and $\vec{v}_2, \vec{v}_3$ are both the first column of the identity matrix. 
In this experiment we set {\tt maxit}$=300$, 
{\tt maxrank}$=10$ (regardless of the grid dimension), while
the tolerance of the stopping criterion to $10^{-3}$.

Table~\ref{tab:ex1} reports CPU time and number of iterations for \tucksscg\ and \tucksssd\,
with its preconditioned variants: both the inner-outer and fast-solver 
strategies are reported, as discussed in section~\ref{sec:prec}. 
The performance for AMEn is also reported, where in parenthesis is the
number of alternating sweeps. 

The preconditioned method with the 
fast-solver shows the best performance, with only two iterations, and this is
clearly due to the fact that the fast solver approximates the same exact operator.
{Both {\pfft } and {\peig } variants are largely superior to AMEn, 
	the others are competitive.}
 We observe that the inner-outer strategy decreases the number of iterations as
hoped for, but is not particularly
effective compared with the other ad-hoc created preconditioners, even giving
higher CPU time than with the unpreconditioned method. 
We found this behavior recurrent also with the other examples, which seems
to imply that better tuning may be necessary for this approach to be competitive.
}
\end{example}

\begin{example} 
{\rm 
The second example consists of the discretization of the operator
\begin{equation}\label{sec:ex2}
    L(u) := ((x+1)(y+1)u_x)_x + ((x+1)(y+1)u_y)_y + ((x+1)(y+1)u_z)_z, 
\end{equation}
with nonconstant separable coefficients.
We define $\mat{D}_x, \mat{D}_y\in\mathbb{R}^{(n+1)\times(n+1)}$ 
the diagonal matrices whose diagonal entries are the coefficient functions evaluated on the mid grid points. The discretized operator is given by
$$
\ttmat{L} = \I \otimes \mat{D}_y \otimes \mat{R}\mat{D}_x\mat{R}^T + \I \otimes \mat{R}\mat{D}_y\mat{R}^T \otimes \mat{D}_x + \mat{R}\mat{R}^T \otimes \mat{D}_y \otimes \mat{D}_x.
$$
We set the input parameters as in Example~\ref{subsec:ex1}.

In spite of the nonconstant coefficients, the results do not significantly
change with respect to the previous example. It is worth mentioning that unpreconditioned
\tucksssd\ in this case behaves very similarly to \tucksscg\ for both dimensions.
It is also interesting to notice that Laplace preconditioning should be close to
being spectrally equivalent to the considered operator \cite{Faber.Manteuffel.Parter.90}, 
and thus we expect that the
number of iterations will not change much as the grid is refined; a more through investigation
with further grid refinements should be considered for a more definite assessment.
We experienced some problems with the single precision version of the \pfft,
which stagnated around the stopping tolerance. We thus used double precision in this
case. In fact, whenever using {\pfft } the number of iterations is so
small that using single precision does not provide any significant benefit.

\begin{table}[h]
	\centering
\caption{Example~\ref{sec:ex2}. 
CPU times in seconds and number of iterations in parenthesis
for different solvers and $n=500, 1000$.  \label{tab:ex2}}
\begin{tabular}{|c|l|l|l|}
\hline
& Algorithm & $n = 500$ & $n = 1000$ \\
\hline
{\tt tol}=$10^{-3}$ &			\tucksscg        	& 0.27 (19)   	&  0.38 (20)   \\
		&			\tucksssd    &   0.22 (20)	&  0.34 (19)  \\
		&\pio\ \tucksssd 	&  0.26 (5)  &   0.61 (6)   \\
		&\pfft\ \tucksssd &  \textbf{0.11} (4)  &  \textbf{0.16} (4)   \\
		&\peig\ \tucksssd & 0.13 (4)  &   0.37 (5)      \\
		&AMEn             	& 0.39 (6*)	& 2.08 (5*)  \\
	\hline

{\tt tol}=$10^{-4}$ &			\tucksscg        	& 1.23 (49)  	&   3.22 (49) \\
		&			\tucksssd        	& 0.94 (53)  	&   2.10 (50) \\
		&\pio\ \tucksssd 	&  1.39 (15)   &   2.25 (15)  \\
		&{\pfft\ \tucksssd} & \textbf{0.35}$^{\dagger}$ (8)  &   \textbf{0.34} (10) \\
		&\peig\ \tucksssd & 0.74 (24)  &     3.53 (29)   \\
		&AMEn             	& 2.24 (6*) 	&    7.06 (6*)\\
			\hline
		\end{tabular}\\
		\begin{flushleft}
			\hspace{60pt}$^*$ number of sweeps
			\hspace{20pt}{$^{\dagger}$ double precision}
		\end{flushleft}
	\end{table}

 The results shown in Table~\ref{tab:ex2} are consistent with those
of Example~\ref{subsec:ex1},  confirming the competitiveness of
preconditioned \tucksssd.
}
\end{example}

\begin{example}\label{ex:diff-react}
{\rm 
This example considers a diffusion-reaction operator 
\begin{equation}\label{ex:ex3}
    L(u) :=  \mu  \Delta u + 10^3 (x+1) (y+1) u,
\end{equation}
discretized in a similar manner as for the previous cases, yielding 
\[
\ttmat{A} = \I \otimes \I \otimes \mu \bm{\Delta}_1 + 
\I \otimes \mu \bm{\Delta}_1 \otimes \I + 
\mu \bm{\Delta}_1 \otimes \I \otimes \I + 
\mat{D}_z \otimes \I \otimes \mat{D}_x.
\]
The input parameters are set as in Example~\ref{subsec:ex1}. The
results are reported in Table~\ref{tab:ex3}.
Once again, when preconditioned, \tucksssd\ is more efficient than AMEn.
}
\end{example}

\begin{table}[h]
	\centering
\caption{Example~\ref{ex:diff-react}. 
CPU time in seconds and number of iterations in parenthesis
for different solvers and $n=500, 1000$. 
\label{tab:ex3}}
\begin{tabular}{|c|l|l|l|}
\hline
& Algorithm & $n = 500$ & $n = 1000$ \\
\hline
{\tt tol}=$10^{-3}$ &			\tucksscg        	& 0.18 (13)   	&  0.27 (16)   \\
		&			\tucksssd    &   0.14 (13)	&  0.23 (17)  \\
		&\pio\ \tucksssd 	&  0.21 (4)  &   0.33 (5)   \\
		&\pfft\ \tucksssd &  \textbf{0.04} (2)  &  \textbf{0.05} (2)   \\
		&\peig\ \tucksssd & 0.08 (2)  &   0.25 (2)      \\
		&AMEn             	& 0.26 (5*)	& 1.03 (5*)  \\
	\hline

{\tt tol}=$10^{-4}$ &			\tucksscg        	& 0.36 (21)  	&   0.72 (30) \\
		&			\tucksssd        	& 0.34 (16)  	&   0.81 (39) \\
		&\pio\ \tucksssd 	&  0.56 (6)   &   0.88 (9)  \\
		&\pfft\ \tucksssd & \textbf{0.06} (3)  &   \textbf{0.10} (4) \\
		&\peig\ \tucksssd & 0.14 (5)  &   0.45 (7)   \\
		&AMEn             	& 0.41 (5*) 	&    2.23 (6*)\\
			\hline
		\end{tabular}\\
		\begin{flushleft}
			\hspace{60pt}$^*$ number of sweeps
		\end{flushleft}
	\end{table}

\begin{example}\label{ex:discontinous}
{\rm 
Finally, we consider an operator with highly discontinuous coefficients, that is
$\nabla \cdot ( a(x)a(y)a(z) \nabla u)$ where 
$$
a(x) = \left \{ \begin{array}{ll}
10^{-2} & {\rm for}\, x\, \in {\cal J}=[\frac 1 4, \frac 3 4] \\
10 & {\rm for} \, x\,  \in [0,1] \setminus {\cal J}
\end{array} \right .
$$
All other parameters are set like in the previous examples. Here we also consider two
different values of the maximum rank for all iterates, that is ${\tt maxrank}=10, 15$.
The results for ${\tt maxrank}=10$
are reported in Table~\ref{tab:ex4-4}, and are consistent with the previous
findings, in spite of the more challenging setting.

In Figure~\ref{fig:conv_exp5_n1000} we also report the convergence history for all methods,
for both values of {\tt maxrank}. We observe that \tucksssd\ slows down significantly as
the number of iterations increases, making it the slowest of all for {\tt maxrank}=10.
For the larger value of {\tt maxrank} more information is retained also for \tucksssd, allowing
the method to closely follow \tucksscg.
We also observe how sensitive the eigenvalue preconditioner is to the truncation 
performed during the run of
\peig\ \tucksssd: while for the first few iterations the convergence
overlaps that of \pfft\ \tucksssd, at a later stage, for smaller residuals, almost
stagnation takes place. This occurs at different stages, depending on the maximum allowed rank,
as expected.
}
\end{example}

\begin{table}[ht]
\centering
\caption{Example~\ref{ex:discontinous}. 
CPU times in seconds and number of iterations in parenthesis
for different solvers and $n=500, 1000$, and {\tt maxrank}=10. 
\label{tab:ex4-4}}
\begin{tabular}{|c|l|l|l|}
\hline
& Algorithm & $n = 500$ & $n = 1000$ \\
\hline
{\tt tol}=$10^{-3}$ &			\tucksscg        	& 0.27 (19)   	&  0.30 (19)   \\
		&			\tucksssd    &   0.22 (19)	&  0.23 (20)  \\
		&\pio\ \tucksssd 	&  0.36 (6)  &   0.38 (6)   \\
		&\pfft\ \tucksssd &  \textbf{0.04} (2)  &  \textbf{0.06} (3)   \\
		&\peig\ \tucksssd & 0.09 (2)  &   0.26 (2)      \\
		&AMEn             	& 0.24 (5*)	& 0.77 (5*)  \\
	\hline
{\tt tol}=$10^{-4}$ &			\tucksscg        	& 1.12 (48)  	&   2.14 (49) \\
		&			\tucksssd        	& 1.52 (87)  	&   1.82 (77) \\
		&\pio\ \tucksssd 	&  1.64 (17)   &   1.78 (18)  \\
		&\pfft\ \tucksssd & \textbf{0.12} (4)  &   \textbf{0.09} (4) \\
		&\peig\ \tucksssd & 0.71 (28)  &   0.96 (23)   \\
		&AMEn             	& 0.78 (6*) 	&    2.30 (6*)\\
			\hline
		\end{tabular}\\
\begin{flushleft}
\hspace{90pt}$^*$ number of sweeps
\end{flushleft}
\end{table}

\begin{figure}[htbp]
\centering
\includegraphics[width=0.49\textwidth]{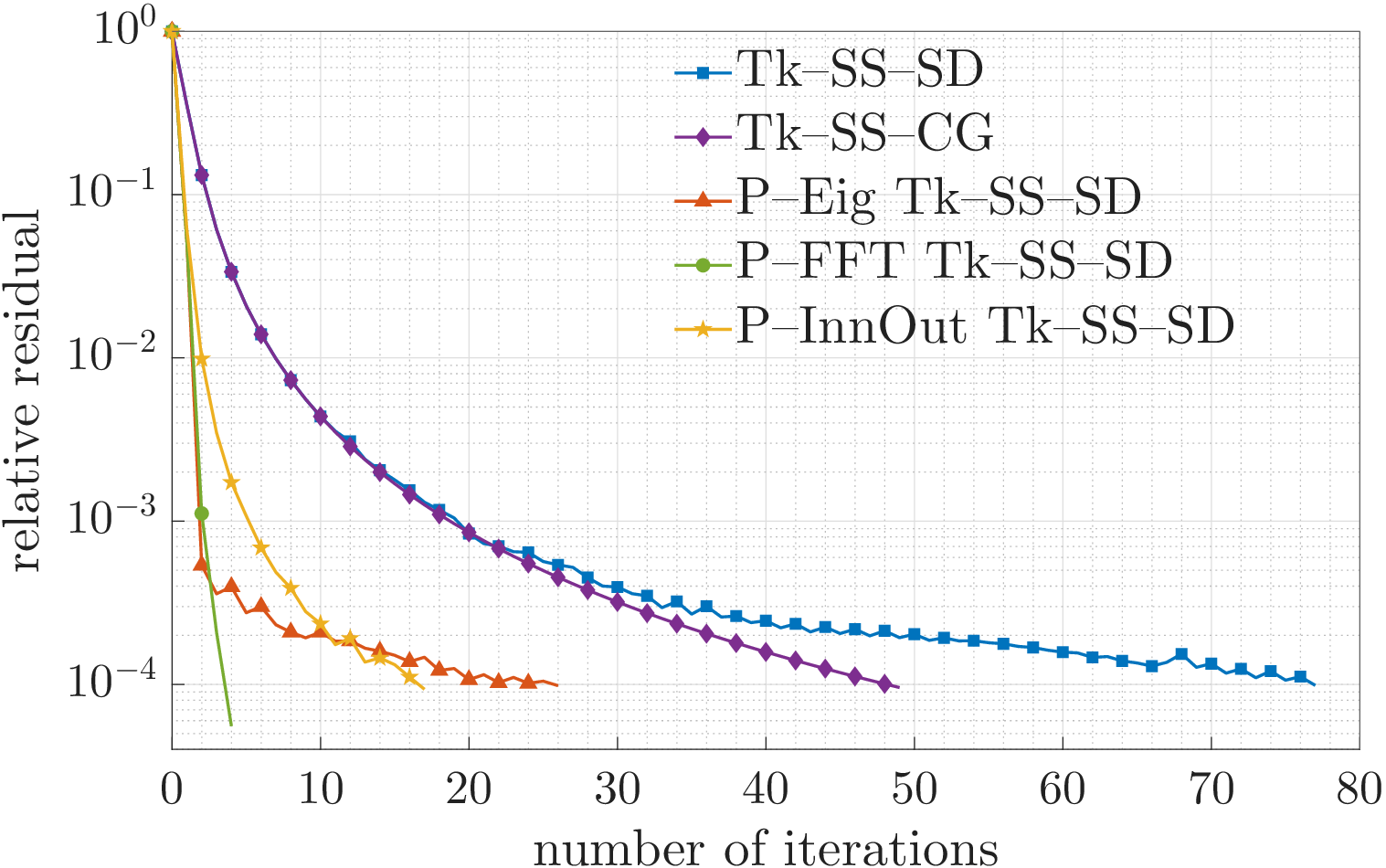} \
\includegraphics[width=0.49\textwidth]{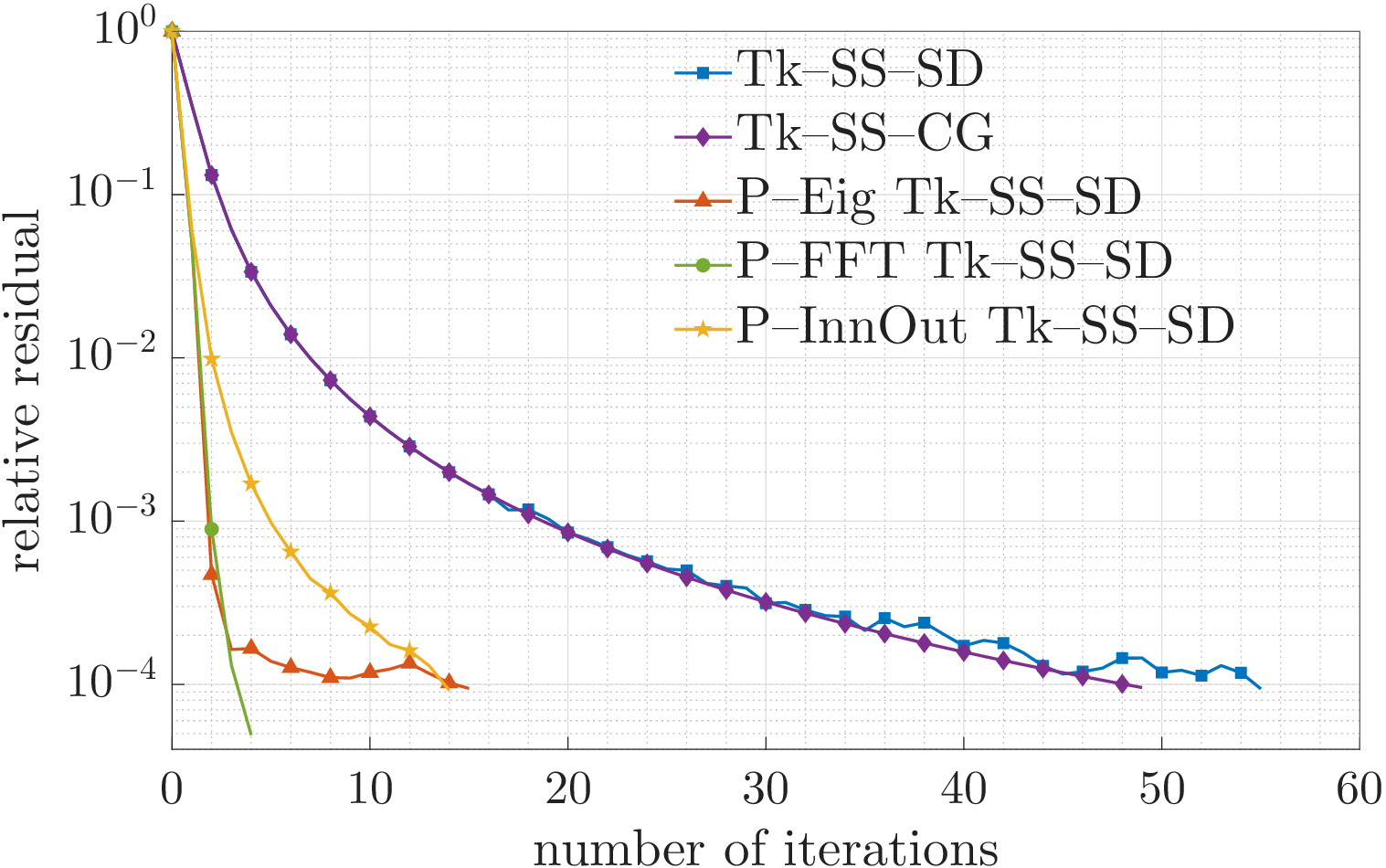}
\caption{Convergence history for $n=1000$, {\tt tol}=$10^{-4}$ and
${\tt maxrank}=10$ (left) and ${\tt maxrank}=15$ (right).}
\label{fig:conv_exp5_n1000}
\end{figure}

\section{Conclusions}\label{sec:conclusions}
We have presented a generalization to tensor equations of the recently introduced
subspace conjugate gradient
 method,  with symmetric and positive definite operators among tensor 
spaces. 
We have also extended the initial idea so as to include a new tensor-based version of the
classical steepest descent procedures, which
naturally adapts to this setting, while somewhat lowering the computational.
To enhance the algorithm performance, we 
have introduced a mixed-precision strategy and we 
have considered different preconditioning strategies. Our numerical experiments 
illustrate that the new class of methods is computationally competitive with 
respect to the well established AMEn.

\section{Acknowledgements}
All authors are members of INdAM, Research Group GNCS.
The work of MI and VS
was partially supported by the
European Union - NextGenerationEU under the National Recovery and
Resilience Plan (PNRR) - Mission 4 Education and research
- Component 2 From research to business - 
Investment 1.1 Notice Prin 2022 - DD N. 104 of 2/2/2022,
entitled “Low-rank Structures and Numerical Methods in Matrix
and Tensor Computations and their
Application”, code 20227PCCKZ – CUP J53D23003620006.
The work of LP is funded by the European Union under the 
National Recovery and Resilience Plan (PNRR) - Mission
4 - Component 2 Investment 1.4 “Strengthening research 
structures and creation of “National R\&D Champions'' on some 
Key Enabling Technologies `` DD N. 3138 of 12/16/2021 rectified 
with DD N. 3175 of 18/12/2021, code CN00000013 - CUP J33C22001170001.

\bibliographystyle{siamplain}
\bibliography{references}

\newpage
\section*{Appendix}
In this appendix we provide a more detailed description of some key implementations.
\begin{algorithm}[H]
	{
		\begin{algorithmic}[1]
			\Statex \textbf{Input:} two order-$d$ tensors in Tucker format, $\ttvec{x} = \mat{F}_1\otimes \cdots \otimes \mat{F}_d\,\ttvec{c}_{\ttvec{x}}$ and $\ttvec{y} = \mat{G}_1\otimes \cdots \otimes \mat{G}_d\,\ttvec{c}_{\ttvec{y}}$, a maximal allowed multilinear rank  of components equal to $\texttt{maxrank}$, a truncating tolerance $\delta$
			\Statex \textbf{Output:} Tucker approximation of the $\ttvec{z} = \ttvec{x} + \ttvec{y}$
			\smallskip
			\State Set $\ttvec{c}_{\ttvec{z}} = \texttt{tenblkdiag}(\ttvec{c}_{\ttvec{x}}, \ttvec{c}_{\ttvec{y}})$
			\For{$j=1,\ldots,d$}\label{alg:tucker-round:l1}
			\State Compute $\mat{U}_j, \mat{\Sigma}_j, \mat{V}_j$ by $\texttt{SVD}([\mat{F}_j, \mat{G}_j])$
			\State Set $t_{j}=\argmax_i{\sigma_i > \delta \|\vec{\sigma}\|_1}$ where $\vec{\sigma} = \text{diag}(\mat{\Sigma}_j)$
			\State Update $\ttvec{c}_{\ttvec{z}}'$ as $\ttvec{c}_{\ttvec{z}}'\times_j (\mat{\Sigma}_j\mat{V}_j^{\top})(\colon, 1\colon t_j)$
			\EndFor\label{alg:tucker-round:l6}
			\State Compute ${\ttvec{c}}'_{\ttvec{z}}, \{\mat{L}_j\}$ from $\texttt{ST-HOSVD}(\ttvec{c}_{\ttvec{z}}, \texttt{maxrank})$\label{alg:tucker-round:l7}
			\State Return ${\ttvec{c}}'_{\ttvec{z}}$ and $\{\mat{U}_j(\colon, 1\colon t_j)\mat{L}_j\}$\label{alg:tucker-round:l8}
		\end{algorithmic}    \caption{Tucker-rounding-sum.\label{alg:tucker-round}}
	}
\end{algorithm}

{\it  Tucker-rounding-sum.}
In terms of computational complexity, the rounding-sum algorithm
performs $d$ SVDs of $(n\times r+p)$ matrices truncated at rank $t$, assuming $n = \max\vec{n}$, $r = \max\vec{r}$, $p = \max\vec{p}$. 
Under the reasonable assumption that $r+ p \ll n$, this step requires $\bigO(dn(r+p)^2)$ flops. 
The $d$ TTM products to form $\ttvec{c}_{\ttvec{z}}$ requires $\bigO(\sum_{i=1}^{d}r^it^{d+1-i})$ flops, knowing that 
$t \le r+p$. 
The ST-HOSVD of $\ttvec{c}_{\ttvec{z}}$ costs $\bigO(t^{d+1})$ flops and the $d$ matrix-matrix product 
to form ${\mat{U}}_j{\mat{L}}_j$ cost $\bigO(dnrt)$ flops. The dominant computational term is given by the $d$ TTM products.

From a storage point of view, 
the largest tensor is $\ttvec{c}_{\ttvec{z}}$ which requires $\bigO((r+p)^d)$ units of memory and is 
the dominant storage cost.

\begin{algorithm}[H]
	{
		\begin{algorithmic}[1]
			\Statex \textbf{Input:}  a tensor operator, $\ttmat{A}$, formed by $\{\mat{A}_{j,h}\}$ for $h=1,\dots,\ell$ and $j=1,\dots,d$, an order-$d$ tensors in Tucker format, $\ttvec{x} = \mat{F}_1\otimes \cdots \otimes \mat{F}_d\,\ttvec{c}_{\ttvec{x}}$, a maximal allowed multilinear rank  of components equal to $\texttt{maxrank}$, a truncating tolerance $\delta$
			\Statex \textbf{Output:} Tucker approximation of the $\ttvec{y} = \ttmat{A}\ttvec{x}$
			\smallskip
			\State Set $\ttvec{c}_{\ttvec{y}} =\ttvec{c}_{\ttvec{x}}$
			\For{$h=1,\ldots,\ell$}
			\State Set $\ttvec{c}_{\ttvec{y}} = \texttt{tenblkdiag}(\ttvec{c}_{\ttvec{x}}, \ttvec{c}_{\ttvec{y}})$
			\State Set $\mat{G}_j = \mat{A}_{j,h}\mat{F}_j$ for $j=1, \dots, d$
			\For{$j=2,\ldots,d$}
			\State Set $\mat{G}_j= [\mat{A}_{j,h}\mat{F}_j, \mat{G}_j]$
			\State Compute $\mat{U}_j, \mat{\Sigma}_j, \mat{V}_j$ by $\texttt{SVD}(\mat{G}_j)$
			\State Set $t_{j}=\argmax_i{\sigma_i > \delta \|\vec{\sigma}\|_1}$ where $\vec{\sigma} = \text{diag}(\mat{\Sigma})$
			\State Update $\mat{G}_j$ as $\mat{U}_j(\colon, 1\colon t_j)$
			
			\State Update $\ttvec{c}_{\ttvec{y}}$ as $\ttvec{c}_{\ttvec{y}}\times_j (\mat{\Sigma}_j\mat{V}_j^{\top})(\colon, 1\colon t_j)$
			\EndFor
			\EndFor
			\State Compute ${\ttvec{c}'}_{\ttvec{y}}, \{\mat{L}_k\}$ from $\texttt{ST-HOSVD}(\ttvec{c}_{\ttvec{y}}, \texttt{maxrank})$
			\State Return ${\ttvec{c}'}_{\ttvec{y}}$ and $\{\mat{U}_j(\colon, 1\colon t_j)\mat{L}_j\}$
		\end{algorithmic}    \caption{Tucker-rounding-product.\label{alg:tucker-roundAx} }
	}
\end{algorithm}

{\it Tucker-rounding-product.}
From a computational complexity point of view, the matrix-matrix products $\mat{A}_{j,h}\mat{F}_j$ are in total $d\ell$, accounting for $\bigO(d\ell nr^2)$ flops, assuming $n = \max\vec{n}$, $r = \max\vec{r}$ and $r\ll n$. The size of $\mat{G}_j$ changes over the iterations, but we can estimate that it never exceeds $r\ell$ thanks to the truncated SVDs at rank $t\le r\ell$. Thus, we perform $d(\ell-1)$ SVDs of $(n\times r\ell)$ matrices, requiring $\bigO(d\ell n(r\ell)^2)$ flops. Each of the $\ell-1$ TTM products to update $\ttvec{c}_{\ttvec{y}}$ along all modes requires $\bigO(\sum_{i=1}^{d}r^it^{d+1-i})$ flops. The total flops for the TTM products can be estimated as $\bigO\bigl(\ell(r\ell)^{d+1}\bigr)$

The ST-HOSVD of $\ttvec{c}_{\ttvec{y}}$ costs $\bigO(t^{d+1})$ flops, which can be estimated as $\bigO((r\ell)^{d+1})$. The $d$ matrix-matrix product $\mat{U}_j(\colon, 1\colon t_j)\mat{L}_j$ cost $\bigO(dnrt)$ flops. The dominant computational term is given by the $(\ell-1)$ TTM products along the $d$ modes.

From a storage point of view, 
the largest tensor is $\texttt{tenblkdiag}(\ttvec{c}_\ttvec{x}, \ttvec{c}_\ttvec{y})$ which requires $\bigO((2r)^d)$ units of memory, representing the dominant storage cost. 

\begin{algorithm}[h]
	{
		\begin{algorithmic}[1]
			\Statex \textbf{Input:} a collection of diagonal matrices $\{\mat{D}_{h}\}_{{h=-q,\dots,q}}$ with $\mat{D}_{h} = {\rm diag}(\exp(t_h\bm{\Lambda}))$ and a coefficient vector $\vec{c}\in\R^{2q+1}$ defining $\bm{\Delta}_d^{-1}$, an order-$d$ tensors in Tucker format, $\ttvec{x} = \mat{U}_1\otimes \cdots \otimes \mat{U}_d\,\ttvec{c}_{\ttvec{x}}$, a maximal allowed multilinear rank  of components equal to $\texttt{maxrank}$
			\Statex \textbf{Output:} Tucker approximation of the $\ttvec{y} = \bm{\Delta}_d^{-1}\ttvec{x}$
			\smallskip
			\For{$j=1,\ldots,d$} 
			\State Set $\widehat{\mat{U}}_j = \texttt{DST-I}(\mat{U}_j)$ 
			\State Set $\mat{E}_j = [\mat{D}_{-q}\widehat{\mat{U}}_j, \dots, \mat{D}_{q}\widehat{\mat{U}}_j]$
			\State Set $\mat{Q}_j, \mat{R}_j =\texttt{QR}(\mat{E}_j)$
			\EndFor
			\State Set $\ttvec{c}'_{\ttvec{x}} = \bm{0}$
			\For{$h=-q,\ldots,q$}
			\State Set $\ttvec{c}'_{\ttvec{x}} = \ttvec{c}'_{\ttvec{x}} + c_h\mat{R}_{1,h}\otimes \mat{R}_{2,h} \otimes\mat{R}_{3,h}\ttvec{c}_{\ttvec{x}}$
			\EndFor
			\State Compute ${\ttvec{c}}_{\ttvec{y}}, \{\mat{L}_j\}$ from $\texttt{ST-HOSVD}(\ttvec{c}'_{\ttvec{x}}, \texttt{maxrank})$
			\State Set ${\mat{W}}_j = \texttt{iDST-I}(\widehat{\mat{U}}_j\mat{L}_j)$ 
			\State Return $\widehat{\ttvec{c}}_{\ttvec{y}}$ and $\{\mat{W}_j\}$
		\end{algorithmic}    \caption{Tucker-FFT-preconditioner.\label{alg:tucker-FTT} }
	}
\end{algorithm}

{\it Fast Fourier Transform preconditioning.}
The FFT-strategy to compute 
$\ttvec{y} = \ttmat{M}\ttvec{x}$ where $\ttvec{x} = \mat{U}_1\otimes\mat{U}_2\otimes\mat{U}_3\,\ttvec{c}_{\ttvec{x}}$ and $\ttvec{y} = \mat{W}_1\otimes\mat{W}_2\otimes\mat{W}_3\,\ttvec{c}_{\ttvec{y}}$
proceeds as follows. 
Firstly, the eigenvalues of $\bm{\Delta}_1$ are computed analytically as
$\vec{\lambda}(i) = 2 - 2\cos(\frac{i\pi}{n+1})$ and so are
the exponential matrices $\diag(\exp(t_h\vec{\lambda}))$ for $h=-q, \dots, q$. 
As mentioned in the text,  to compute $\ttvec{y} = \ttmat{M}\ttvec{x}$ in Tucker format
the Tucker factors of $\ttvec{x}$ are transformed with the DST-I,
i.e., $\widehat{\mat{U}}_{j} = \texttt{DST-I}(\mat{U}_j)$.
 
Let $\bm{\Lambda} = \text{diag}(\vec{\lambda})$, 
then to compute $\ttvec{x} = \exp\bigl(t_h(\bm{\Lambda} \oplus \bm{\Lambda}\oplus\bm{\Lambda}\bigr)\widehat{\ttvec{x}}$ 
we form the $3$ block matrices 
$\mat{E}_j = [\mat{E}_{j,-q}, \dots, \mat{E}_{j, q}]$ 
of size $(n\times (2q+1)r_j)$ where $\mat{E}_{j,h} = \diag(\exp(t_h\vec{\lambda}))\widehat{\mat{U}}_j$ for 
$j=1,\dots,3$ and $h=-q, \dots,q$. 
The matrices $\mat{E}_j$ are factorized by the 
{QR} algorithm, that is $\mat{E}_j = \mat{Q}_j\mat{R}_j$, 
where $\mat{Q}_j$ is of size $(n \times p_j)$ with $p_j = {\rm rank}(\mat{E}_j)$, and  $\mat{R}_j = [\mat{R}_{j,-q}, \dots, \mat{R}_{j, q}]$. 
The blocks of $\mat{R}_j$ are used to update the core of $\ttvec{x}$. 
In details, 
\[
	\ttvec{c}'_{\ttvec{x}} = \sum_{h= -q}^{q} c_h \mat{R}_{1,h}\otimes \mat{R}_{2,h} \otimes\mat{R}_{3,h}\ttvec{c}_{\ttvec{x}},
\]
which is a dense tensor of size $\vec{p}$. To limit memory usage, the dense tensor $\ttvec{c}'_{\ttvec{y}}$ is approximated at multilinear-rank equal to $\texttt{maxrank}$ in all its component by ST-HOSVD, obtaining ${\ttvec{c}}'_{\ttvec{x}} = \mat{L}_1\otimes \cdots \otimes \mat{L}_d\, {\ttvec{c}}_{\ttvec{y}}$, where ${\ttvec{c}}_{\ttvec{y}}\in\R^{\vec{m}}$, and $\mat{L}_h \in\R^{p_j  \times m_j}$ with $m_j  \le \texttt{maxrank}$ for $j  = 1, \dots, d$. 
  
Finally,  $\ttvec{y} = \ttmat{V}_h{\ttvec{x}'}$ is computed in Tucker-format using the iDST-I. 
In particular, we let the Tucker factors of $\ttvec{y}$ be
 $\mat{W}_j = \texttt{iDST-I}(\widehat{\mat{U}}_j\mat{L}_j)$ for  $j=1,\dots,3$, and its Tucker core 
be $\ttvec{c}_{\ttvec{y}}$, previously computed by ST-HOSVD. 
The entire procedure is summarized in Algorithm~\ref{alg:tucker-FTT}.

From the computational complexity viewpoint, the precomputation phase 
requires $\bigO(n)$ flops. Thus, assembling all parts
is cheaper than in the first algebraic approach, as the explicit eigenvectors formation
is omitted. 
Assuming now $d$ modes, applying the DST-I and iDST-I requires 
$\bigO(drn\log n)$ flops, respectively. 
Forming $d$ matrices $\mat{E}_j$ requires $\bigO(dqn)$ flops.
To assess the computational complexity of the {QR} factorization, let us
 assume that $(2q+1)r \ll n$ where $n=\max\vec{n}$ and $r=\max{\vec{r}}$. Then, 
the $d$ \texttt{QR} factorization needs $\bigO\bigl(dn(2qr^2)\bigr)$ flops. The computation of the $(2q+1)$ TTM products along all modes to form $\ttvec{c}'_{\ttvec{x}}$ costs $\bigO(\sum_{i=1}^{d}p^ir^{d+1-i})$ flops and its ST-HOSVD costs $\bigO(p^{d+1})$ flops. 
Assuming that $p\le 2qr$, the TTM products and the ST-HOSVD require 
$\bigO(2q(2qr)^{d+1})$ and $\bigO((2qr)^{d+1})$. Lastly, forming the inputs 
for the iDST-I requires $\bigO(dnrm)$ flops. Overall, the dominant cost is given by
the $2q+1$ TTM products along each of the $d$ modes. 

In terms of storage, the key difference with the purely algebraic
 strategy is given by the use of the $d$ matrices $\mat{E}_j$ of size $(n\times (2q+1)r_j)$. 
Indeed, in the eigenvalue-factorization strategy the \texttt{tucker-rounding-product} algorithm is
employed, which truncates the blocks every time they are concatenated, never forming such large matrices.  
On the other hand, this multiple truncation necessarily causes loss of information, that may 
cause the first preconditioner to perform less well.

\end{document}